\documentclass[a4paper,12pt]{article}
\usepackage{amscd,amssymb,amsmath,amsthm}
\usepackage{hyperref}
\usepackage{graphicx,caption,dsfont}
\usepackage{subcaption}
\usepackage{amsfonts,enumerate}
\usepackage{fullpage}
\usepackage[numbers,sort & compress]{natbib}
\usepackage{soul}
\usepackage{xcolor}
\usepackage{epstopdf}

\newtheorem{thm}{Theorem}[section]

\newtheorem{defn}[thm]{Definition}

\newtheorem{rem}{Remark}

\newtheorem{nt}{Note}

\begin{document}
\baselineskip=0.71cm
\begin{center}
{\Large \textbf{Invariant subspaces and exact solutions for some types of scalar and coupled time-space fractional diffusion equations}}
\end{center}
\vspace{-0.5cm}
\begin{center}
P. Prakash\\
Department of Mathematics,\\
  Amrita Vishwa Vidyapeetham,\\ Coimbatore-641112, INDIA.\\
E-mail ID: vishnuindia89@gmail.com.
\end{center}
\begin{abstract}
 We explain how the invariant subspace method can be extended to a scalar and coupled system of time-space fractional partial differential equations. The effectiveness and applicability of the method have been illustrated through time-space (i) fractional diffusion-convection equation, (ii) fractional reaction-diffusion equation, (iii) fractional diffusion equation with source term, (iv) two-coupled system of fractional diffusion equation, (v) two-coupled system of fractional stationary transonic plane-parallel gas flow equation and (vi) three-coupled system of fractional Hirota-Satsuma KdV equation. Also, we explicitly presented how to derive more than one exact solution of the equations as mentioned above using the invariant subspace method.
\end{abstract}
\textbf{Key-words}\\
Time-space fractional PDEs, Invariant subspace method, Laplace transformation technique, Mittag-Leffler function.
\section{Introduction}
The subject of fractional calculus is one of the most rapidly developing areas of mathematical analysis.
The study of fractional differential equations (FDEs) has considerable popularity and importance during the past few decades, mainly due to their widespread applications in various fields of science and engineering such as fluid flow, viscoelasticity, aerodynamics, electromagnetic theory, rheology, signal processing, electrical networks and so on \cite{pi,kai,kts,rh,db, mo11,aad}. 
 In the last few decades, several analytical and numerical techniques have been developed to construct exact and numerical solutions of nonlinear differential equations.
However, the derivation of the exact solution of FDEs is not an easy task, because some properties of fractional derivatives are harder than the classical derivative.

For this reason, in recent years, both mathematicians and physicists have been paid much attention to study the exact and numerical solutions of nonlinear fractional partial differential equations (FPDEs) using various ad hoc methods, such as Lie group analysis method \cite{pra1,kdv,pra2,pra5,ba1,jef}, Adomian decomposition method \cite{ad1,Ge,mo}, homotopy decomposition method \cite{ad3}, differential transform method \cite{mo2}, function-expansion method \cite{ru1,ru2,ru3}  and so on. However, recent investigations have shown that a new analytic method based on the invariant subspace approach provides an effective tool to derive the exact solution of scalar and coupled system of time-space FPDEs. This method was originally developed by Galaktionov and Svirshchevskii \cite{gs} (see also \cite{wy,wx,wx1,qz,sp,jp,non,li,ps1}) for PDEs and was further extended by Gazizov and Kasatkin \cite{rk} (see also \cite{ra,pa1,pa2,oe1,s1,hes,pra3,pra4,ne1,sc,ps2}) for time FPDEs.

The main objective of this article is to demonstrate how the invariant subspace method provides an effective tool to derive exact solution of the following time-space FPDEs namely (i) time-space fractional diffusion-convection equation, (ii) time-space fractional reaction-diffusion equation, (iii) time-space fractional diffusion equation with source term and (iv) two-coupled system of time-space fractional diffusion equation.

Here we would like to point out that only a limited number of applications for the coupled system of time-space FPDEs have been investigated through the invariant subspace method. We also explain the applicability and effectiveness of the method have been illustrated through time-space fractional
 (i) two-coupled system of diffusion equation, (ii) two-coupled system of stationary transonic plane-parallel gas flow equation \cite{jef,gs} and (iii) three-coupled system of Hirota-Satsuma KdV equation \cite{pra2} and derived their exact solution.

The layout of this paper is as follows: In section 2, some basic concepts of fractional calculus and a brief details of the invariant subspace method for scalar and $m$-component coupled system of nonlinear time-space FPDEs in the sense of Riemann-Liouville/Caputo fractional derivative are presented. In section 3, the effectiveness of the method is illustrated by solving the above-mentioned scalar and coupled system of time-space FPDEs.
Finally, a summary of our results is given in section 4.
\section{Preliminaries}
 In this section, we would like to present some basic definitions and results related to the fractional calculus. Also, we present brief details of the invariant subspace method for scalar and coupled system of time-space FPDEs.
\begin{defn}[\cite{pi,kai}]
The Riemann-Liouville (R-L) fractional derivative of order $\alpha>0$ of the function $g \in L^{1}([a,b],\mathds{R_{+}})$ is defined by
\begin{equation*}
\dfrac{^{RL}d^{\alpha}g(t)}{dt^{\alpha}}=\left\{
                                            \begin{array}{ll}
                                              \dfrac{1}{\Gamma(n-\alpha)}\dfrac{d^{n}}{dt^{n}}\left(\int\limits^{t}_{0}\dfrac{g(s)}{(t-s)^{\alpha-n+1}}ds\right), & n-1<\alpha<n; \\
                                              g^{(n)}(t), & \alpha=n,\ n\in\mathbb{N}.
                                            \end{array}
                                          \right.
\end{equation*}
\end{defn}
\begin{nt}[\cite{pi,kai}]
The R-L fractional derivative of $g(t)=t^{\mu}$ is as follows
\begin{eqnarray}
\label{pp1}^{RL}\dfrac{d^\alpha t^\mu}{dt^{\alpha}}=\dfrac{\Gamma(\mu+1)}{\Gamma(\mu-\alpha+1)}t^{\mu-\alpha},~\alpha>0,~\mu>-1,~t>0.
\end{eqnarray}
\end{nt}
\begin{defn}[\cite{pi,kai}]
The Caputo fractional derivative of order $\alpha>0$ of the function $g\in C^{n}([a,b])$ is defined by
\begin{eqnarray}\label{zz1}
\nonumber \dfrac{d^{\alpha}g(t)}{dt^{\alpha}}
=\left\{
                                                               \begin{array}{ll}
                                                                 \dfrac{1}{\Gamma(n-\alpha)}\int\limits^{t}_{0}\dfrac{g^{(n)}(s)}{(t-s)^{\alpha-n+1}}ds, & n-1<\alpha< n;\\
                                                                 g^{(n)}(t), & \alpha=n,\ n\in\mathds{N}.
                                                               \end{array}
                                                             \right.
\end{eqnarray}
\end{defn}
\begin{nt}[\cite{pi,kai}]
The Caputo fractional derivative of $g(t)=t^{\mu}$ is as follows
\begin{eqnarray}
\label{pp2}\dfrac{d^\alpha t^\mu}{dt^{\alpha}}=\left\{
                                     \begin{array}{lp{1cm}}
                                       0, & \hbox{if $\mu\in\{0,1,\ldots,n-1\},~  n = [\alpha]+1$;} \\
                                       \dfrac{\Gamma(\mu+1)}{\Gamma(\mu-\alpha+1)}t^{\mu-\alpha}, & \hbox{if $\mu\in\mathbb{N}$ \& $\mu\geq n$  or  $\mu\notin \mathbb{N}$ \& $\mu>n-1.$}
                                     \end{array}
                                   \right.
\end{eqnarray}
\end{nt}
\begin{nt}[\cite{pi,kai}]
The Laplace transformation of Caputo fractional derivative of order $\alpha\in(n-1, n], n\in\mathds{N},$ is
\begin{equation*}
\mathrm{L}\left\{\dfrac{d^{\alpha}g(t)}{dt^{\alpha}}\right\}=s^{\alpha}\overline{g}(s)-\sum\limits^{n-1}_{k=0}s^{\alpha-k-1}g^{(k)}(0), \  \mathcal{R}e(s)>0.
\end{equation*}
\end{nt}
\begin{defn} \cite{pi,kai}
Two parametric Mittag-Leffler function is defined as
\begin{equation*}
\mathbf{E}_{\beta,\gamma}(z)=\sum\limits^{\infty}_{r=0}\dfrac{z^{r}}{\Gamma(\beta r+\gamma)},\ \beta,\gamma, z\in\mathbb{C}, \ \mathcal{R}e(\beta)>0, \mathcal{R}e(\gamma)>0.
\end{equation*}
\end{defn}
Some properties of the Mittag-Leffler function are as follows:
\begin{eqnarray*}
\label{mi1}\mathbf{E}_{1,1}(z)&=&e^{z},\\
\label{mi3}\mathbf{E}_{1,k}(z)&=&\dfrac{1}{z^{k-1}}\left[e^{z}-\sum\limits^{k-2}_{r=0}\dfrac{z^{r}}{r!}\right], \ k\in\mathbb{N}.
\end{eqnarray*}
Note that $$\mathbf{E}_{\beta,1}(z)\equiv\mathbf{E}_{\beta}(z).$$
\begin{nt} [\cite{pi,kai}]
The Laplace transformation of $t^{\gamma-1}\mathbf{E}_{\beta,\gamma}\left(\pm kt^{\beta}\right)$ is
\begin{eqnarray*}
\mathrm{L}\left\{t^{\gamma-1}\mathbf{E}_{\beta,\gamma}\left(\pm kt^{\beta}\right)\right\}=\dfrac{s^{\beta-\gamma}}{(s^{\beta}\mp k)}, \ \mathcal{R}e(s)>|k|^\frac{1}{\beta}.
\end{eqnarray*}
\end{nt}

\textbf{Caputo fractional derivative of the Mittag-Leffler functions:}
\begin{eqnarray*}
\dfrac{d^{\alpha}}{dt^{\alpha}}\left[t^{\gamma-1}\mathbf{E}_{\beta,\gamma}\left(k t^{\beta}\right)\right]&=&t^{\gamma-\alpha-1}\mathbf{E}_{\beta,\gamma-\alpha}\left(k t^{\beta}\right),\\
\dfrac{d^{\alpha}}{dt^{\alpha}}\left[\mathbf{E}_{\alpha}(kt^{\alpha})\right]&=&k\mathbf{E}_{\alpha}(kt^{\alpha}),
\end{eqnarray*}
where $\alpha,\beta,\gamma>0$ and $k\in\mathds{R}$.
\begin{thm}
If $L\left\{\varphi(t)\right\}=\bar{\varphi}(s)$ and $L\left\{\phi(t)\right\}=\bar{\phi}(s)$, then
\begin{eqnarray*}
\varphi(t)\star\phi(t)=\int\limits^{t}_{0}\varphi(t-\tau)\phi(\tau)d\tau=L^{-1}\left\{\bar{\varphi}(s)\bar{\phi}(s)\right\},
\end{eqnarray*}
where $\varphi(t)\star\phi(t)$ is called a convolution of $\varphi(t)$ and $\phi(t)$.
\end{thm}
\subsection{Invariant subspace method for scalar and coupled FPDEs}
\subsubsection{Scalar time-space FPDE}
We consider the following generalized scalar time-space FPDE
\begin{eqnarray}\label{teq1}
\begin{aligned}
\sum_{i=0}^{m}\lambda_{i}\dfrac{\partial^{\alpha+i}u(x,t)}{\partial t^{\alpha+i}}&=\hat{G}[u(x,t)]\\
&=G\left[x,u,\dfrac{\partial^{\beta}u}{\partial x^{\beta}},\dots,\dfrac{\partial^{\beta}}{\partial x^{\beta}}\left(\dfrac{\partial^{\beta}u}{\partial x^{\beta}}\right),\dfrac{\partial^{r\beta}u}{\partial x^{r\beta}},\dfrac{\partial^{\beta+k}u}{\partial x^{\beta+k}}\right],\ \alpha, \beta>0,\, k,r\in\mathbb{N},
\end{aligned}
\end{eqnarray}
$\lambda_{i}\in\mathbb{R}$, where $\hat{G}[u]$ is a linear/nonlinear fractional differential operator. Here, $\dfrac{\partial^{\beta}}{\partial x^{\beta}}(.)$ and $\dfrac{\partial^{\alpha}}{\partial t^{\alpha}}(.)$ are space and time fractional derivatives in the R-L/Caputo sense and $\hat{G}(.)$ is a sufficiently given smooth function.
First, we define the linear space
\begin{eqnarray*}
\mathcal{W}_{n}&=&\mathfrak{L}\left\{\varphi_{1}(x),\dots,\varphi_{n}(x)\right\}\\
&=&\left\{\sum_{j=1}^{n}a_{j}\varphi_{j}(x)\large{|}\ a_{j}\in\mathbb{R},\ j=1,2,\dots,n\right\},
\end{eqnarray*}
where the functions $\varphi_{1}(x),\dots,\varphi_{n}(x)$ are linearly independent. The linear space $\mathcal{W}_n$ is said to be invariant with respect to the fractional differential operator $\hat{G}[u]$ if $\hat{G}:\mathcal{W}_{n}\rightarrow\mathcal{W}_{n}$, that is $\hat{G}[\mathcal{W}_n]\subseteq \mathcal{W}_n$ or $\hat{G}[u]\in\mathcal{W}_{n}$, for all $u\in\mathcal{W}_{n}$. This means that there exist $n$-functions $\Psi_{1}$, $\Psi_{2}$,$\dots$,$\Psi_{n}$ such that
\begin{equation*}
\hat{G}\left[\sum^{n}_{j=1}a_{j}\varphi_{j}(x)\right]=\sum^{n}_{j=1}\Psi_{j}\left(a_{1},a_{2},\dots,a_{n}\right)\varphi_{j}(x),\  \text{for}  \ a_{j}\in\mathds{R}.
\end{equation*}
\begin{thm}
Let $\mathcal{W}_{n}$ be an $n$-dimensional linear space over $\mathbb{R}$. If $\mathcal{W}_{n}$
is invariant under the fractional differential operator $\hat{G}[u]$, then the time-space FPDE \eqref{teq1} admits the following exact solution
\begin{equation}\label{teq2}
u(x,t)=A_{1}(t)\varphi_{1}(x)+A_{2}(t)\varphi_{2}(x)+\dots+A_{n}(t)\varphi_{n}(x),
\end{equation}
where the coefficients $A_{j}(t)$, $(j=1,2,\dots,n)$ satisfy the following system of FODEs
\begin{equation}
\sum^{m}_{i=0}\lambda_{i}\dfrac{d^{\alpha+i}A_{j}(t)}{dt^{\alpha+i}}=\Phi_{j}(A_{1}(t),A_{2}(t),\dots,A_{n}(t)),\ j=1,\dots,n.
\end{equation}
\end{thm}
\begin{proof}
  Using the linearity of the fractional derivative with equation \eqref{teq2}, we obtain
  \begin{eqnarray}
  \begin{aligned}\label{teq5}
  \sum_{i=0}^{m}\lambda_{i}\dfrac{\partial^{\alpha+i}u(x,t)}{\partial t^{\alpha+i}}
  =\sum_{j=1}^{n}\left[\sum_{i=0}^{m}\lambda_{i}\dfrac{d^{\alpha+i}A_{j}(t)}{dt^{\alpha+i}}\right]\varphi_{j}(x).
  \end{aligned}
  \end{eqnarray}
  Let $\mathcal{W}_{n}$ be an invariant subspace with respect to the fractional differential operator $\hat{G}[u]$. Then there exist $n$ functions $\Phi_{1},\Phi_{2},\dots,\Phi_{n}$ such that
  \begin{equation}\label{teq4}
  \hat{G}\left[\sum^{n}_{j=1}a_{j}\varphi_{j}(x)\right]=\sum_{j=1}^{n}\Phi_{j}(a_{1},a_{2},\dots,a_{n})\varphi_{j}(x), \text{for}\ a_{j}\in\mathbb{R},
  \end{equation}
  where $\Phi_{j}$'s are expansion coefficients of $\hat{G}[u]\in\mathcal{W}_{n}$ corresponding to $\varphi_{j}$'s.
  From equations \eqref{teq2} and \eqref{teq4}, we have
  \begin{eqnarray}
  \begin{aligned}\label{teq3}
  \hat{G}[u(x,t)]=&\hat{G}\left[\sum_{j=1}^{n}A_{j}(t)\varphi_{j}(x)\right]\\
  =&\sum^{n}_{j=1}\Phi_{j}(A_{1}(t),\dots,A_{n}(t))\varphi_{j}(x).
  \end{aligned}
  \end{eqnarray}
  Substituting equations \eqref{teq3} and \eqref{teq5} in equation \eqref{teq1}, we have
  \begin{equation}\label{eq6}
  \sum_{j=1}^{n}\left[\sum_{i=0}^{m}\lambda_{i}\dfrac{d^{\alpha+i}A_{j}}{dt^{\alpha+i}}-\Phi_{j}(A_{1}(t),A_{2}(t),\dots,A_{n}(t))\right]\varphi_{j}(x)=0.
  \end{equation}
From equation \eqref{eq6} and using their linear independence of $\left\{\varphi_{j}(x),\ j=1,2,\dots,n\right\}$, we yield the system of FODEs
  \begin{equation}
  \sum^{m}_{i=0}\lambda_{i}\dfrac{d^{\alpha+i}A_{j}(t)}{dt^{\alpha+i}}=\Phi_{j}(A_{1}(t),A_{2}(t),\dots,A_{n}(t)),\ j=1,2,\dots,n.
  \end{equation}
\end{proof}
\subsubsection{Two-coupled system of time-space FPDEs}
Consider the following two-coupled system of time-space FPDEs
\begin{eqnarray}
\begin{aligned}\label{cct1}
\dfrac{\partial^{\alpha_{1}}u_{1}}{\partial t^{\alpha_{1}}}&=G_{1}\left(x,u_{1},u_{2},\dfrac{\partial^{\beta}u_{1}}{\partial x^{\beta}},\dfrac{\partial^{\beta}u_{2}}{\partial x^{\beta}},\dots,\dfrac{\partial^{r\beta}u_{1}}{\partial x^{r\beta}},\dfrac{\partial^{r\beta}u_{2}}{\partial x^{r\beta}},\dfrac{\partial^{\beta+k_{1}}u_{1}}{\partial x^{\beta+k_{1}}},\dfrac{\partial^{\beta+k_{1}}u_{2}}{\partial x^{\beta+k_{1}}}\right),\\
\dfrac{\partial^{\alpha_{2}}u_{2}}{\partial t^{\alpha_{2}}}&=G_{2}\left(x,u_{1},u_{2},\dfrac{\partial^{\beta}u_{1}}{\partial x^{\beta}},\dfrac{\partial^{\beta}u_{2}}{\partial x^{\beta}},\dots,\dfrac{\partial^{r\beta}u_{1}}{\partial x^{r\beta}},\dfrac{\partial^{r\beta}u_{2}}{\partial x^{r\beta}},\dfrac{\partial^{\beta+k_{2}}u_{1}}{\partial x^{\beta+k_{2}}},\dfrac{\partial^{\beta+k_{2}}u_{2}}{\partial x^{\beta+k_{2}}}\right),
\end{aligned}
\end{eqnarray}
$\alpha_{1},\alpha_{2},\beta>0$,$k_{1},k_{2},r\in\mathbb{N}$, where $G_{1}$, $G_{2}$ are generalized linear/nonlinear fractional differential operators and can be considered as given sufficient smooth functions, and $\dfrac{\partial^{\alpha}}{\partial t^{\alpha}}(.)$ and $\dfrac{\partial^{\beta}}{\partial x^{\beta}}(.)$ are time and space fractional derivatives in R-L/Caputo sense. Hereafter, we will use the following notations throughout the article
\begin{equation*}
\hat{G}_{p}[u_{1},u_{2}]=G_{p}\left(x,u_{1},u_{2},\dfrac{\partial^{\beta}u_{1}}{\partial x^{\beta}},\dfrac{\partial^{\beta}u_{2}}{\partial x^{\beta}},\dots,\dfrac{\partial^{\beta}}{\partial x^{\beta}}\left(\dfrac{\partial^{\beta}u_{1}}{\partial x^{\beta}}\right),\dfrac{\partial^{r\beta}u_{1}}{\partial x^{r\beta}},\dfrac{\partial^{r\beta}u_{2}}{\partial x^{r\beta}},\dfrac{\partial^{\beta+k_{p}}u_{1}}{\partial x^{\beta+k_{p}}},\dfrac{\partial^{\beta+k_{p}}u_{2}}{\partial x^{\beta+k_{p}}}\right),
\end{equation*}
 $u_{p}=u_{p}(x,t)$, $p=1,2.$\\
 \textbf{Estimation of invariant subspace:}
  Following the above similar procedure for scalar time-space FPDEs, we develop the following result for the two-coupled system of time-space FPDEs. First, we define the linear spaces $$\mathcal{W}_{n_{p}}^{p}=\mathfrak{L}\left\{\varphi_{1}^{p}(x),\dots,\varphi_{n_{p}}^{p}(x)\right\}\equiv\left\{\sum\limits_{j=1}^{n_{p}}a_{j}^{p}\varphi_{j}^{p}(x)\ \Big|\ a_{j}^{p}\in\mathbb{R},\ j=1,\dots,n_{p}\right\},\  p=1,2,$$  where the functions $\varphi_{1}^{p}(x),\dots,\varphi_{n_{p}}^{p}(x)$ are linearly independent. The linear spaces  $\mathcal{W}_{n_{p}}^{p}$, $p=1,2$, are called an invariant under the vector fractional differential operator $\hat{\mathbb{G}}=\left({G}_{1},{G}_{2}\right)$ if
$\hat{\mathbb{G}}:\mathcal{W}_{n_{1}}^{1}\times\mathcal{W}_{n_{2}}^{2}\rightarrow\mathcal{W}_{n_{1}}^{1}\times\mathcal{W}_{n_{2}}^{2}$, which means that
$\hat{G}_{p}:\mathcal{W}_{n_{1}}^{1}\times\mathcal{W}_{n_{2}}^{2}\rightarrow\mathcal{W}_{n_{p}}^{p}, \ p=1,2.$, that is, $\hat{G}_{p}\left[\mathcal{W}_{n_{1}}^{1}\times\mathcal{W}_{n_{2}}^{2}\right]\subseteq\mathcal{W}_{n_{p}}^{p}$ or $\hat{G}_{p}[u_1,u_2]\in\mathcal{W}_{n_{p}}^{p}$, for all $(u_1,u_2)\in\mathcal{W}_{n_{1}}^{1}\times\mathcal{W}_{n_{2}}^{2}$, $p=1,2$. Then, we have
$$\hat{G}_{p}\left[\sum\limits_{j=1}^{n_{1}}a_{j}^{1}\varphi_{j}^{1}(x),\sum\limits_{j=1}^{n_{2}}a_{j}^{2}\varphi_{j}^{2}(x)\right]=\sum\limits_{j=1}^{n_{p}}\Psi_{j}^{p}
\left(a_{1}^{1},\dots,a_{n_{1}}^{1},a_{1}^{2},\dots,a_{n_{2}}^{2}\right)\varphi_{j}^{p}(x),\ p=1,2.$$
\begin{thm}
Let $\mathcal{W}^{p}_{n_{p}}$  be
 a finite dimensional linear space over $\mathbb{R}$. If $\mathcal{W}^{p}_{n_{p}}$
is invariant with respect to the fractional differential operator $\hat{G}_{p}[u_{1},u_{2}]$, then the two-coupled system of time-space FPDE \eqref{cct1} admits the following exact solution
\begin{equation}\label{cct2}
u_{p}(x,t)=\sum^{n_{p}}_{j=1}A_{j}^{p}(t)\varphi^{p}_{j}(x),\ p=1,2,
\end{equation}
where the coefficients $A^{p}_{j}(t)$ satisfy the following system of FODEs
\begin{equation}
\dfrac{d^{\alpha_{p}}A_{j}(t)}{dt^{\alpha_{p}}}=\Phi_{j}^{p}(A_{1}^{1}(t),A_{2}^{1}(t),\dots,A_{n_{1}}^{1}(t),A_{1}^{2}(t),\dots,A_{n_{2}}^{2}(t)),\ j=1,\dots,n_{p},\ p=1,2.
\end{equation}
\end{thm}
\begin{proof}
  Using the linearity of the fractional derivative with equation \eqref{cct2}, we obtain
  \begin{eqnarray}
  \begin{aligned}\label{cct5}
  \dfrac{\partial^{\alpha_{p}}u_{p}(x,t)}{\partial t^{\alpha_{p}}}
  =\sum_{j=1}^{_{n_{p}}}\dfrac{d^{\alpha_{p}}A_{j}(t)}{dt^{\alpha_{p}}}\varphi^{p}_{j}(x),\ p=1,2.
  \end{aligned}
  \end{eqnarray}
   Let $\mathcal{W}_{n_{p}}^{p}$ be an invariant subspace under the fractional differential operator $\hat{G}[u_{1},u_{2}]$. Then there exists the functions $\Phi^{p}_{1},\Phi^{p}_{2},\dots,\Phi^{p}_{n_{p}}$, $(p=1,2)$ such that
  \begin{equation}\label{cct4}
  \hat{G}_{p}\left[\sum^{n_{1}}_{j=1}a^{1}_{j}\varphi^{1}_{j}(x),\sum^{n_{2}}_{j=1}a^{2}_{j}\varphi_{j}^{2}(x)\right]=\sum_{j=1}^{n_{p}}\Phi_{j}^{p}(a^{1}_{1},\dots,a^{1}_{n_{1}},a^{2}_{1},\dots,a^{2}_{n_{2}})\varphi_{j}^{p}(x),\ \ a_{j}^{p}\in\mathbb{R},\ p=1,2,
  \end{equation}
  where $\Phi^{p}_{j}$'s are expansion coefficients of $\hat{G}[u_{1},u_{2}]\in\mathcal{W}^{p}_{n_{p}}$ corresponding to $\varphi^{p}_{j}$'s.
  From equations \eqref{cct2} and \eqref{cct4}, we have
  \begin{eqnarray}
  \begin{aligned}\label{cct3}
  \hat{G}_{p}[u_{1}(x,t),u_{2}(x,t)]=&\hat{G}_{p}\left[\sum_{j=1}^{n_{1}}A^{1}_{j}(t)\varphi^{1}_{j}(x),\sum_{j=1}^{n_{2}}A_{j}^{2}(t)\varphi_{j}^{2}(x)\right]\\
  =&\sum^{n_{p}}_{j=1}\Phi_{j}^{p}(A_{1}^{1}(t),\dots,A_{n_{1}}^{1}(t),A_{1}^{2}(t),\dots,A_{n_{2}}^{2})\varphi_{j}(x),\ p=1,2.
  \end{aligned}
  \end{eqnarray}
  Substituting equations \eqref{cct3} and \eqref{cct5} in equation \eqref{cct1}, we have
  \begin{equation}\label{cct6}
  \sum_{j=1}^{n_{p}}\left[\dfrac{d^{\alpha_{p}}A_{j}^{p}}{dt^{\alpha_{p}}}-\Phi_{j}^{p}(A_{1}^{1}(t),A_{2}^{1}(t),\dots,A_{n_{1}}^{1}(t),A_{1}^{2}(t),A_{2}^{2}(t),\dots,A_{n_{2}}^{2}(t))\right]\varphi_{j}^{p}(x)=0,\ p=1,2.
  \end{equation}
  By their linear independence of $\left\{\varphi^{p}_{j},\ j=1,2,\dots,n_{p},\ p=1,2\right\}$, we have
  \begin{equation}
  \dfrac{d^{\alpha_{p}}A^{p}_{j}(t)}{dt^{\alpha_{p}}}=\Phi(A_{1}^{1}(t),A_{2}^{1}(t),\dots,A_{n}^{1}(t),A_{1}^{2}(t),A_{2}^{2}(t),\dots,A_{n}^{2}(t)),\ j=1,2,\dots,n_{p},\ p=1,2.
  \end{equation}
\end{proof}
 \subsubsection{$m$-coupled system of time-space FPDEs}
 Consider the following $m$-coupled system of time-space FPDEs
 \begin{equation}\label{ccm}
 \dfrac{\partial^{\alpha_{p}}\mathds{U}}{\partial t^{\alpha_{p}}}=\hat{\mathbb{G}}(\mathds{U})\equiv\left(G_{1}(\mathds{U}),\dots,G_{m}(\mathds{U})\right)\in\mathbb{R}^{m},\alpha_{p}>0,\ p=1,2,\dots,m,
 \end{equation}
 where the operators $G_{q}(.)\ (q=1,2,\dots,m)$ are generalized linear/nonlinear fractional differential operators and can be considered as sufficient smooth functions, and $\dfrac{\partial^{\alpha}}{\partial t^{\alpha}}(.)$ and $\dfrac{\partial^{\beta}}{\partial x^{\beta}}(.)$ are time and space fractional derivatives in R-L/Caputo sense, and $\mathbb{U}=(u_{1},u_{2},\dots,u_{m})\in\mathbb{R}^{m}$, $u_p=u_p(x,t)$,
 \begin{equation*}
\hat{G}_{p}[\mathds{U}]=G_{p}\left(x,u_{1},\dots,u_{m},\dfrac{\partial^{\beta}u_{1}}{\partial x^{\beta}},\dots,\dfrac{\partial^{\beta}u_{2}}{\partial x^{\beta}},\dfrac{\partial^{r\beta}u_{1}}{\partial x^{r\beta}},\dots,\dfrac{\partial^{r\beta}u_{m}}{\partial x^{r\beta}},\dots,\dfrac{\partial^{\beta+k_{p}}u_{p}}{\partial x^{\beta+k_{p}}},\dfrac{\partial^{\beta+k_{p}}u_{m}}{\partial x^{\beta+k_{p}}}\right),\
\end{equation*}
$r,k_{p}\in\mathbb{N}$, $\beta>0$, $p=1,2,\dots,m.$\\
 \textbf{Estimation of invariant subspace:}
  Proceeding the above similar procedure, we can develop the following result for an $m$-coupled system of time-space FPDEs. Here, we define the linear spaces $$\mathcal{W}_{n_{p}}^{p}=\mathfrak{L}\left\{\varphi_{1}^{p}(x),\dots,\varphi_{n_{p}}^{p}(x)\right\}\equiv\left\{\sum\limits_{j=1}^{n_{p}}a_{j}^{p}\varphi_{j}^{p}(x)\ \Big|\ (a_{j}^{p},\dots,a_{n_{p}}^{p})\in\mathds{R}^{n_{p}}\right\},\ p=1,2,\dots,m,$$
 where the functions $\varphi_{1}^{p}(x),\dots,\varphi_{n_{p}}^{p}(x)$  $(n_{p}\geq1)$ are linearly independent. The linear spaces  $\mathcal{W}_{n_{p}}^{p}$, $p=1,2,\dots,m$, are called an invariant under the vector fractional differential operator $\hat{\mathbb{G}}=\left({G}_{1},{G}_{2},\dots,G_{m}\right)$ if
 $\hat{\mathbb{G}}:\mathcal{W}_{n_{1}}^{1}\times\dots\times\mathcal{W}_{n_{m}}^{m}\rightarrow\mathcal{W}_{n_{1}}^{1}\times\dots\times\mathcal{W}_{n_{m}}^{m}$, which means that $\hat{G}_{p}:\mathcal{W}_{n_{1}}^{1}\times\dots\times\mathcal{W}_{n_{m}}^{m}\rightarrow\mathcal{W}_{n_{p}}^{p}, \ p=1,2,\dots,m$, that is, $\hat{G}_{p}\left[\mathcal{W}_{n_{1}}^{1}\times\dots\times\mathcal{W}_{n_{m}}^{m}\right]\subseteq\mathcal{W}_{n_{p}}^{p}$ or $\hat{G}_{p}[u_1,\dots,u_m]\in\mathcal{W}_{n_{p}}^{p}$, for all $(u_1,\dots,u_m)\in\mathcal{W}_{n_{1}}^{1}\times\dots\times\mathcal{W}_{n_{m}}^{m}$, $p=1,\dots,m$.
Then there exists $\Phi_{j}^{p}$, $j=1,2,\dots,n_{p}$, $p=1,2,\dots,m$, such that
$$\hat{G}_{p}\left[\sum\limits_{j=1}^{n_{1}}a_{j}^{1}\varphi_{j}^{1}(x),\dots,\sum\limits_{j=1}^{n_{m}}a_{j}^{m}\varphi_{j}^{m}(x)\right]=\sum\limits_{j=1}^{n_{p}}\Phi_{j}^{p}
\left(a_{1}^{1},\dots,a_{n_{1}}^{1},\dots,a_{1}^{m},\dots,a_{n_{m}}^{m}\right)\varphi_{j}^{p}(x),$$
for all $(a_{1}^{p},\dots,a_{n_{p}}^{p})\in\mathds{R}^{n_{p}}$, $p=1,2,\dots,m$.
\begin{thm}
Let $\mathcal{W}^{p}_{n_{p}}$ be a finite dimensional linear space over $\mathbb{R}$
and if $\mathcal{W}^{p}_{n_{p}}$ is invariant under the fractional differential operator $\hat{G}_{p}[\mathds{U}]$, then the $m$-coupled system \eqref{ccm} has a solution of the form
\begin{equation}\label{cctt2}
u_{p}(x,t)=\sum^{n_{p}}_{j=1}A_{j}^{p}(t)\varphi^{p}_{j}(x),\ p=1,2,\dots,m,
\end{equation}
where the coefficients $A^{p}_{j}(t)$ satisfy the following system of FODEs
\begin{equation}
\dfrac{d^{\alpha_{p}}A_{j}(t)}{dt^{\alpha_{p}}}=\Phi_{j}(A_{1}^{1}(t),A_{2}^{2}(t),\dots,A_{n_{2}}^{2}(t)),\ j=1,\dots,n_{p},\ p=1,2,\dots,m.
\end{equation}
\end{thm}
\begin{proof}
Similar to the proof of theorem 2.6.
\end{proof}
 Let us assume that invariant subspace $W_{n_{p}}^{p}=\mathfrak{L}\left\{\varphi_{1}^{p},\dots,\varphi_{n_{p}}^{p}\right\}$ is defined as space generated by solutions of the following linear fractional order ODEs
\begin{equation*}
L_{p}[y_{p}]=y_{p}^{(\alpha)}+c_{n_{p}-1}^{p}(x)y_{p}^{(\alpha-1)}+\dots+c_{0}^{p}(x)y_{p}=0,\ n_{p}-1<\alpha\leq n_{p},\ n_{p}\in\mathbb{N},\ p=1,2,\dots,m,
\end{equation*}
where $y_{p}^{(\alpha)}=\dfrac{d^{\alpha}y_{p}}{dx^{\alpha}}$. Thus the invariant condition reads
\begin{equation*}
L_{p}\left[\hat{G}_{p}[\mathds{U}]\right]|_{[H_{1}]\cap\dots\cap[H_{m}]}=0, p=1,2,\dots,m,
\end{equation*}
where $[H_{p}]$ denotes the equation $L_{p}[u_{p}]=0$ and its differential consequences with respect to $x$.
\section{Construction of invariant subspaces and exact solutions}
\subsection{Time-space fractional diffusion-convection equation}
Consider the following time-space fractional diffusion-convection equation
\begin{equation}
\label{doce}\dfrac{\partial^{\alpha}u}{\partial t^{\alpha}}=\hat{G}[u]=\left(\dfrac{\partial^{\beta} u}{\partial x^{\beta}}\right)^{2}\left(\dfrac{\partial p}{\partial u}\right)+p(u)\dfrac{\partial^{\beta}}{\partial x^{\beta}}\left(\dfrac{\partial^{\beta}u}{\partial x^{\beta}}\right)-\dfrac{\partial^{\beta} u}{\partial x^{\beta}}\left(\dfrac{\partial q}{\partial u}\right), \ t>0, \ \ \alpha,\beta\in(0,1],
\end{equation}
where the functions $p(u)$ and $q(u)$ represent the phenomenon of diffusion and convection respectively. 
 The above PDE with $\alpha=1$ and $\beta=1$ was discussed through invariant subspace method in \cite{li,pra4}. We would like to point out that the operator $\hat{G}[u]$ admits no invariant subspace for arbitrary functions $p(u)$ and $q(u)$. Hence, we choose $p(u)=a_{n}u^{n}+a_{n-1}u^{n-1}+\dots+a_{1}u+a_{0}$ and $q(u)=b_{n+1}u^{n+1}+b_{n}u^{n}+\dots+b_{1}u+b_{0}$, $n\in\mathbb{N}$,
where $a_{n}$, $a_{n-1},\dots,a_{0},b_{n+1},\dots,b_{1},b_{0}$ are arbitrary constants.\\
Then, the equation \eqref{doce} reduces to
\begin{eqnarray}
\begin{aligned}
\label{E2}\dfrac{\partial^{\alpha}u}{\partial t^{\alpha}}=\hat{G}[u]=&\left[na_{n}u^{n-1}+(n-1)a_{n-1}u^{n-2}+\dots+a_{1}\right]\left(\dfrac{\partial^{\beta} u}{\partial x^{\beta}}\right)^{2}\\
&+\left[a_{n}u^{n}+a_{n-1}u^{n-1}+\dots+a_{1}u+a_{0}\right]\dfrac{\partial^{\beta}}{\partial x^{\beta}}\left(\dfrac{\partial^{\beta}u}{\partial x^{\beta}}\right)\\
&-\left[(n+1)b_{n+1}u^{n}+nb_{n}u^{n-1}+\dots+b_{1}\right]\dfrac{\partial^{\beta} u}{\partial x^{\beta}}, \ t>0, \ \ \alpha,\beta\in(0,1].
\end{aligned}
\end{eqnarray}
It is easy to find that the differential operator $\hat{G}[u]$ admits a one-dimensional invariant subspace $\mathcal{W}_{1}=\mathfrak{L}\left\{E_{\beta}(kx^{\beta})\right\}$, $k\in\mathbb{R}$, if $a_{r}k=b_{r+1}$, $r=1,2\dots n$, $n\in\mathbb{N}$, because
$$\hat{G}\left[A_{1}E_{\beta}(kx^{\beta})\right]=\left(a_{0}k^{2}-b_{1}k\right)A_{1} E_{\beta}(kx^{\beta})\in\mathcal{W}_{1}.$$
Thus, we can write the exact solution in the form
\begin{equation}\label{E3}
u(x,t)=A_{1}(t)E_{\beta}(kx^{\beta}),
\end{equation}
where $A_{1}(t)$ is an unknown function to be determined. Substituting \eqref{E3} in \eqref{E2}, we get
\begin{equation}\label{E4}
\dfrac{d^{\alpha}A_{1}}{dt^{\alpha}}=(a_{0}k^{2}-b_{1}k)A_{1}(t).
\end{equation}
First, we consider $\alpha=\beta=1$. In this case, we have
\begin{equation}\label{E6}
u(x,t)=k_{0}e^{(a_{0}k^{2}-b_{1}k)t+kx},\ k_{0},k,a_{0},b_{1}\in\mathbb{R}.
\end{equation}
Next, we consider $\alpha,\ \beta\in(0,1]$. Applying Laplace transformation technique on both sides of equation \eqref{E4}, we get
$$s^{\alpha}\bar{A_{1}}(s)-s^{\alpha-1}A_{1}(0)=(a_{0}k^{2}-b_{1}k)\bar{A_{1}}(s)$$
which can be written as
$$\bar{A_{1}}(s)=\dfrac{k_{0}s^{\alpha-1}}{s^{\alpha}-(a_{0}k^{2}-b_{1}k)},\ \text{where}\ ~ k_{0}=A_{1}(0).$$
Applying inverse Laplace transformation to the above equation, we get
\begin{equation*}
A_{1}(t)=k_{0} E_{\alpha}((a_{0}k^{2}-b_{1}k)t^{\alpha}).
\end{equation*}
Hence, we obtain an exact solution for time-space fractional diffusion-convection equation \eqref{E2} as follows
\begin{equation}\label{E5}
u(x,t)=k_{0}E_{\alpha}((a_{0}k^{2}-b_{1}k)t^{\alpha})E_{\beta}(kx^{\beta}),\ \alpha,\beta\in(0,1], k_{0},k,a_{0},b_{1}\in\mathbb{R}.
\end{equation}
Note that, for $\alpha=\beta=1$, equation \eqref{E5} is exactly same as \eqref{E6}. The above exact solution \eqref{E5} for $a_0=2$, $k=k_0=b_1=1$, $t=1$, and different values of $\alpha$ and $\beta$ is shown in Fig. (a).
\begin{rem}
Let $p(u)=a_{1}u+a_{0}$, $q(u)=-ka_{1}u^{2}+b_{1}u+b_{0}$, $k,a_{0},a_{1},b_{1},b_{0}\in\mathbb{R}$ and $k(\neq0)$.
Then, the equation \eqref{doce} can be written as follows
\begin{equation}\label{FE1}
\dfrac{\partial^{\alpha}u}{\partial t^{\alpha}}=a_{1}\left(\dfrac{\partial^{\beta}u}{\partial x^{\beta}}\right)^{2}+(a_{1}u+a_{0})\dfrac{\partial^{\beta}}{\partial x^{\beta}}\left(\dfrac{\partial^{\beta}u}{\partial x^{\beta}}\right)+(2ka_{1}u-b_{1})\dfrac{\partial^{\beta}u}{\partial x^{\beta}}.
\end{equation}
It is easy to find that equation \eqref{FE1} admits a two-dimensional invariant subspace $\mathcal{W}_{2}=\mathfrak{L}\left\{1,E_{\beta}(-kx^{\beta})\right\}$. Thus, we can write the exact solution of equation \eqref{FE1} as follows
\begin{equation}\label{FE2}
u(x,t)=A_{1}(t)+A_{2}(t)E_{\beta}(-kx^{\beta}),
\end{equation}
where the coefficients $A_1(t)$ and $A_{2}(t)$ are satisfy the following system of FODEs
\begin{eqnarray}
\begin{aligned}
\label{FE3}&\dfrac{d^{\alpha}A_{1}}{dt^{\alpha}}=0,\\
&\dfrac{d^{\alpha}A_{2}}{dt^{\alpha}}=(-k^{2}a_{1}A_{1}+a_{0}k^{2}+kb_{1})A_{2}.
\end{aligned}
\end{eqnarray}
First, we consider $\alpha=\beta=1$. Solving the above system \eqref{FE3}, we have
\begin{equation}\label{FE5}
u(x,t)=k_{1}+k_{2}e^{k((-kk_{1}a_{1}+a_{0}k+b_{1})t-x)},\ k,k_{1},k_{2},a_{0},a_{1},b_{1}\in\mathbb{R}.
\end{equation}
 Next, we consider $\alpha,\beta\in(0,1]$. Applying Laplace transformation technique to the above system \eqref{FE3}, we yield an exact solution of equation \eqref{FE1} as follows
\begin{equation}\label{FE6}
u(x,t)=k_{1}+k_{2}E_{\alpha}\left((-k^{2}a_{1}k_{1}+a_{0}k^{2}+kb_{1})t^{\alpha}\right)E_{\beta}(-kx^{\beta}),\ k,k_{1},k_{2},a_{0},a_{1},b_{1}\in\mathbb{R}.
\end{equation}
 Observe that, for $\alpha=\beta=1$, equation \eqref{FE6} is exactly same as \eqref{FE5}. The above exact solution \eqref{FE6} for $k_1=-1$, $a_1=2$, $a_{0}=0$, $k_{2}=k=b_1=1$, $t=2$, and different values of $\alpha$ and $\beta$ is shown in Fig. (b).
\end{rem}

\begin{figure}[h!]
\begin{center}
\begin{subfigure}[]{0.80\textwidth}
  \includegraphics[width=\textwidth]{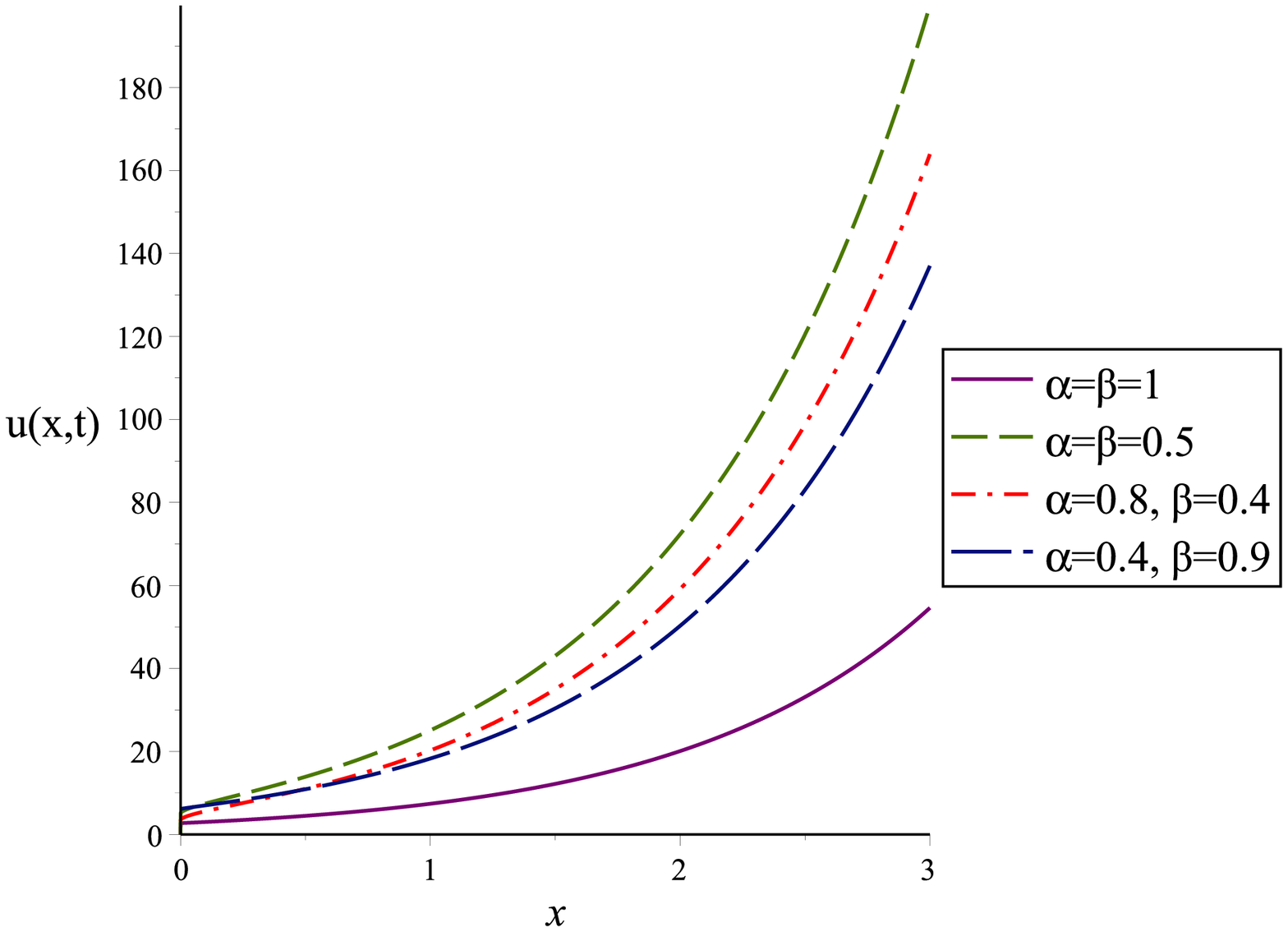}
\footnotesize{Fig.(a) Graphical representation of the solution \eqref{E5} for $a_0=2$, $k=k_0=b_1=1$, $t=1$, and different values of $\alpha$ and $\beta$.}
 \end{subfigure}\hspace{30pt}
 \begin{subfigure}[]{0.80\textwidth}
  \includegraphics[width=\textwidth]{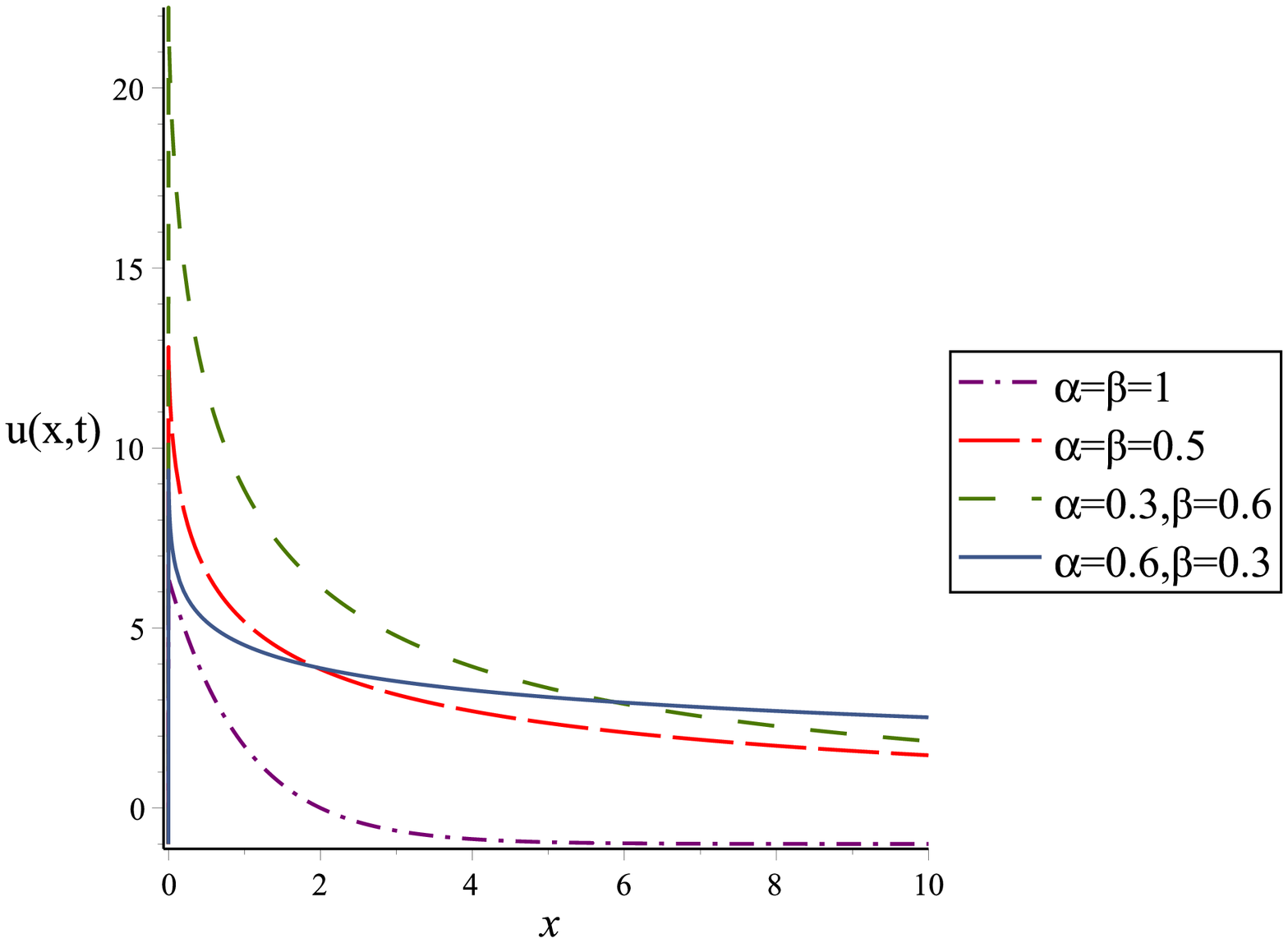}
 \footnotesize{Fig.(b) Graphical representation of the solution \eqref{FE6} for $k_1=-1$, $a_1=2$, $a_{0}=0$, $k_{2}=k=b_1=1$, $t=2$, and different values of $\alpha$ and $\beta$.}
 \end{subfigure}
\end{center}
\end{figure}

\begin{rem}
Let $p(u)=a_{0}$ and $q(u)=b_{1}u+b_{0}$, $a_{0},b_{1},b_{0}\in\mathbb{R}$.\\
Then, equation \eqref{doce} reduces to
\begin{equation}\label{RE1}
\dfrac{\partial^{\alpha}u}{\partial t^{\alpha}}=a_{0}\dfrac{\partial^{\beta}}{\partial x^{\beta}}\left(\dfrac{\partial^{\beta}u}{\partial x^{\beta}}\right)-b_{1}\dfrac{\partial^{\beta}u}{\partial x^{\beta}}
\end{equation}
which admits the following distinct invariant subspaces\\
$(i)\ \mathcal{W}_{2}=\mathfrak{L}\left\{1,x^{\beta}\right\}.$\\
$(ii)\ \mathcal{W}_{n}=\mathfrak{L}\left\{E_{\beta}(k_{1}x^{\beta}),\dots,E_{\beta}(k_{n}x^{\beta})\right\}$, $n\in\mathbb{N}$,\\
$(iii)\ \mathcal{W}_{n+1}=\mathfrak{L}\left\{1,E_{\beta}(k_{1}x^{\beta}),\dots,E_{\beta}(k_{n}x^{\beta})\right\}$, $n\in\mathbb{N},$\\
$(iv)\ \mathcal{W}_{n+2}=\mathfrak{L}\left\{1,x^{\beta},E_{\beta}(k_{1}x^{\beta}),\dots,E_{\beta}(k_{n}x^{\beta})\right\}$, $n\in\mathbb{N}$,\ $k_{i}\in\mathbb{R},\ i=1,\dots,n$.\\
First, we consider the invariant subspace $\mathcal{W}_{n}=\mathfrak{L}\left\{E_{\beta}(k_{1}x^{\beta}),\dots,E_{\beta}(k_{n}x^{\beta})\right\}$, $n\in\mathbb{N}$. Thus, we can write the exact solution in the form
\begin{equation}\label{RE2}
u(x,t)=A_{1}(t)E_{\beta}(k_{1}x^{\beta})+\dots+A_{n}(t)E_{\beta}(k_{n}x^{\beta}),
\end{equation}
where $A_{i}(t)$, $i=1,\dots,n$, are satisfy the following system of $n$-FODEs
\begin{eqnarray}
\begin{aligned}\label{ss}
&\dfrac{d^{\alpha}A_{1}}{dt^{\alpha}}=\left(a_{0}k_{1}^{2}-b_{1}k_{1}\right)A_{1},\\
&\vdots\\
&\dfrac{d^{\alpha}A_{n}}{dt^{\alpha}}=\left(a_{0}k_{n}^{2}-b_{1}k_{n}\right)A_{n}.
\end{aligned}
\end{eqnarray}
Applying Laplace transformation technique to the above system \eqref{ss}, we obtain an exact solution of equation \eqref{RE1} with $\alpha=\beta=1$, reads
\begin{equation}\label{RE3}
u(x,t)=\sum\limits^{n}_{s=1}r_{s}e^{((a_{0}k_{s}-b_{1})k_{s}t+k_{s}x)},
\end{equation}
while $\alpha\in(0,1]$, it takes
\begin{equation}\label{RE4}
u(x,t)=\sum\limits^{n}_{s=1}r_{s}E_{\alpha}\left((a_{0}k_{s}-b_{1})k_{s}t^{\alpha}\right)E_{\beta}(k_{s}x^{\beta}),\ r_{s}, k_{s}, a_{0}, b_{1}\in\mathbb{R}\ (s=1,2,\dots,n).
\end{equation}
We observe that for $\alpha=\beta=1$, equation \eqref{RE4} is exactly same as \eqref{RE3}.\\
Proceeding the above similar procedure, we can derive another more general exact solution associated with the more general invariant subspace $\mathcal{W}_{n+2}=\mathfrak{L}\left\{1,x^{\beta},E_{\beta}(k_{1}x^{\beta}),\dots,E_{\beta}(k_{n}x^{\beta})\right\}$, $n\in\mathbb{N}$. For this case, we obtain the more general exact solution of \eqref{RE1} reads
\begin{equation}\label{REE4}
u(x,t)=c_{1}-c_{2}b_{1}\dfrac{\Gamma(\beta+1)}{\Gamma(\alpha+1)}t^{\alpha}+c_{2}x^{\beta}+\sum\limits^{n}_{s=1}r_{s}E_{\alpha}\left((a_{0}k_{s}-b_{1})k_{s}t^{\alpha}\right)E_{\beta}(k_{s}x^{\beta}),\ \alpha,\beta\in(0,1],
\end{equation}
where $c_{1}$, $c_{2}$, $r_{s}$, $k_{s}$, $(s=1,2,\dots,n)$, $a_{0}$ and $b_{1}$ are arbitrary constants. Observe that, for $c_{1}=c_{2}=0$, equation \eqref{REE4} is exactly the same as \eqref{RE4}. Similarly, we can derive different types of exact solutions for time-space fractional diffusion-convection equation \eqref{RE1} using the other above mentioned invariant subspaces.
\end{rem}
\begin{rem}
Let $p(u)=a_{1}u$ and $q(u)=\dfrac{b_{2}}{2}u^{2}$, $b_{2}$ and $a_{1}$ are constants.\\
Then, the equation \eqref{doce} reduces into
\begin{equation}\label{RE5}
\dfrac{\partial^{\alpha}u}{\partial t^{\alpha}}=a_{1}\left(\dfrac{\partial^{\beta}u}{\partial x^{\beta}}\right)^{2}+a_{1}u \dfrac{\partial^{\beta}}{\partial x^{\beta}}\left(\dfrac{\partial^{\beta}u}{\partial x^{\beta}}\right)-b_{2}u\dfrac{\partial^{\beta}u}{\partial x^{\beta}}.
\end{equation}
It is easy to find that equation \eqref{RE5} admits a two-dimensional invariant subspace $\mathcal{W}_{2}=\mathfrak{L}\left\{1,E_{\beta}(kx^{\beta})\right\}$ if $a_{1}=\dfrac{b_{2}}{2k}.$
Hence, we obtain an exact solution of \eqref{RE5} with $\alpha=\beta=1$, reads
\begin{equation}\label{RE7}
u(x,t)=k_{0}+k_{1}e^{k(-\frac{b_{2}}{2}k_{0}t+x)},
\end{equation}
while $\alpha\in(0,1]$, it takes
\begin{equation}\label{RE8}
u(x,t)=k_{0}+k_{1}E_{\alpha}\left(-\frac{b_{2}}{2}kk_{0}t^{\alpha}\right)E_{\beta}(kx^{\beta}),\ k_{0},\ k_{1},\ b_{2},\ k\in\mathbb{R}.
\end{equation}
The above exact solution \eqref{RE8} for $k=k_0=k_1=b_2=1$, $t=2$, and different values of $\alpha$ and $\beta$ is shown in Fig. (c).
Note that, for $\alpha=\beta=1$, equation \eqref{RE8} is exactly same as \eqref{RE7}. We would like to point out that when $\beta=1$, the above solution \eqref{RE8} is exactly same as given in \cite{pra4}.
\end{rem}
\begin{rem}
Let $p(u)=u$ and $q(u)=b_{0}$, $b_0\in\mathbb{R}$.\\
Then, equation \eqref{doce} describes only the time-space fractional diffusion equation as follows
\begin{eqnarray}\label{RE9}
\dfrac{\partial^{\alpha}u}{\partial t^{\alpha}}=\left(\dfrac{\partial^{\beta}u}{\partial x^{\beta}}\right)^{2}+u \dfrac{\partial^{\beta}}{\partial x^{\beta}}\left(\dfrac{\partial^{\beta}u}{\partial x^{\beta}}\right)
\end{eqnarray}
which admits invariant subspace $\mathcal{W}_{2}=\mathfrak{L}\left\{1,x^{\beta}\right\}$. In this case, we have
\begin{equation}\label{RPP}
u(x,t)=k_{0}+k_{1}^{2}\dfrac{(\Gamma(\beta+1))^{2}}{\Gamma(\alpha+1)}t^{\alpha}+k_{1}x^{\beta},\ k_0,k_1\in\mathbb{R},\ \alpha,\beta\in(0,1].
\end{equation}
The above exact solution \eqref{RPP} for $k_0=k_1=1$, $t=2$, and different values of $\alpha$ and $\beta$ is shown in Fig. (d). We would like to mention that when $\beta=1$, the above solution \eqref{RPP} is exactly the same as given in \cite{pra4}.
\end{rem}
\begin{figure}[h!]
\begin{center}
\begin{subfigure}[]{0.80\textwidth}
  \includegraphics[width=\textwidth]{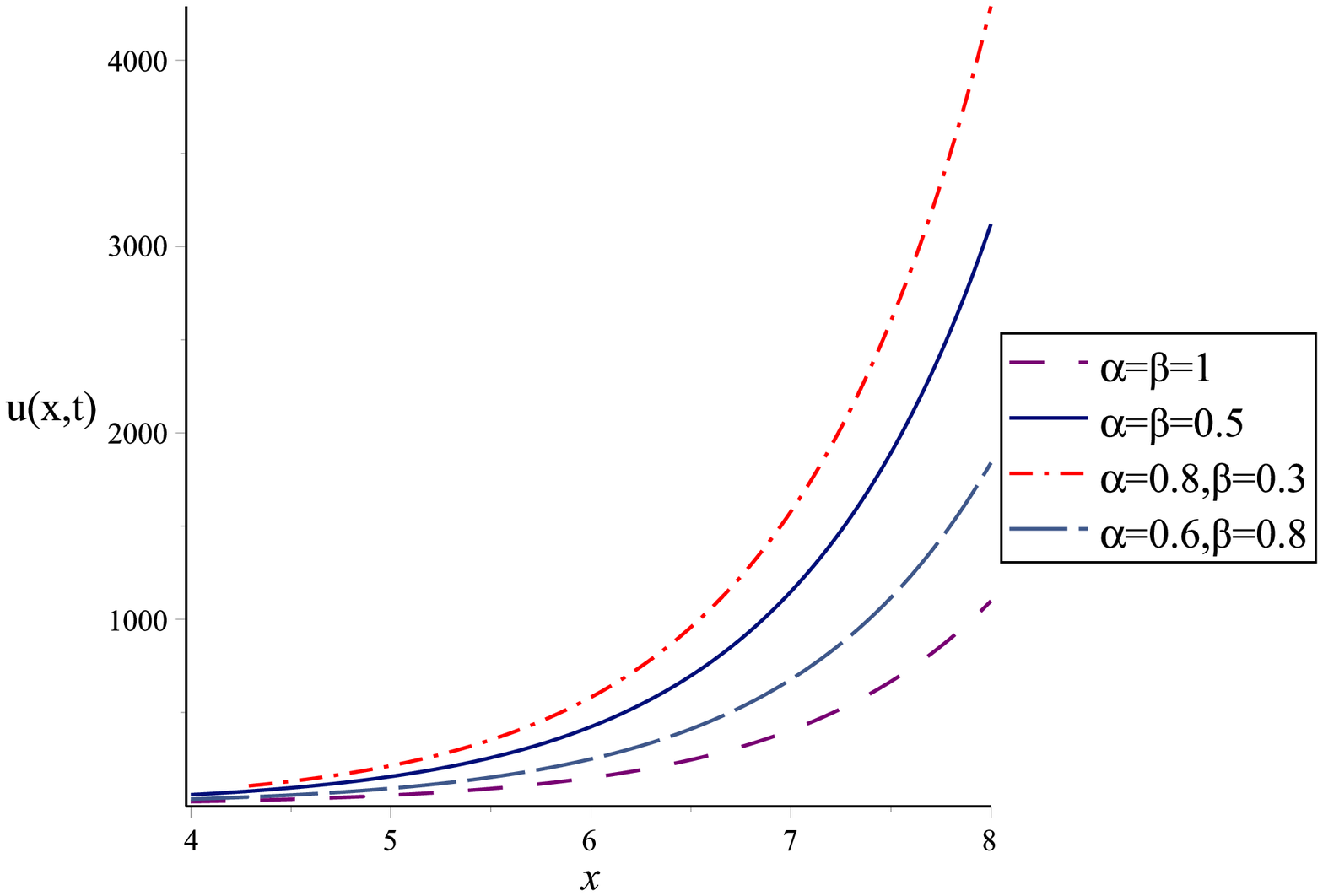}
\footnotesize{Fig.(c) Graphical representation of the solution \eqref{RE8} for $k=k_0=k_1=b_2=1$, $t=2$, and different values of $\alpha$ and $\beta$.}
 \end{subfigure}\hspace{10pt}
 \begin{subfigure}[]{0.80\textwidth}
  \includegraphics[width=\textwidth]{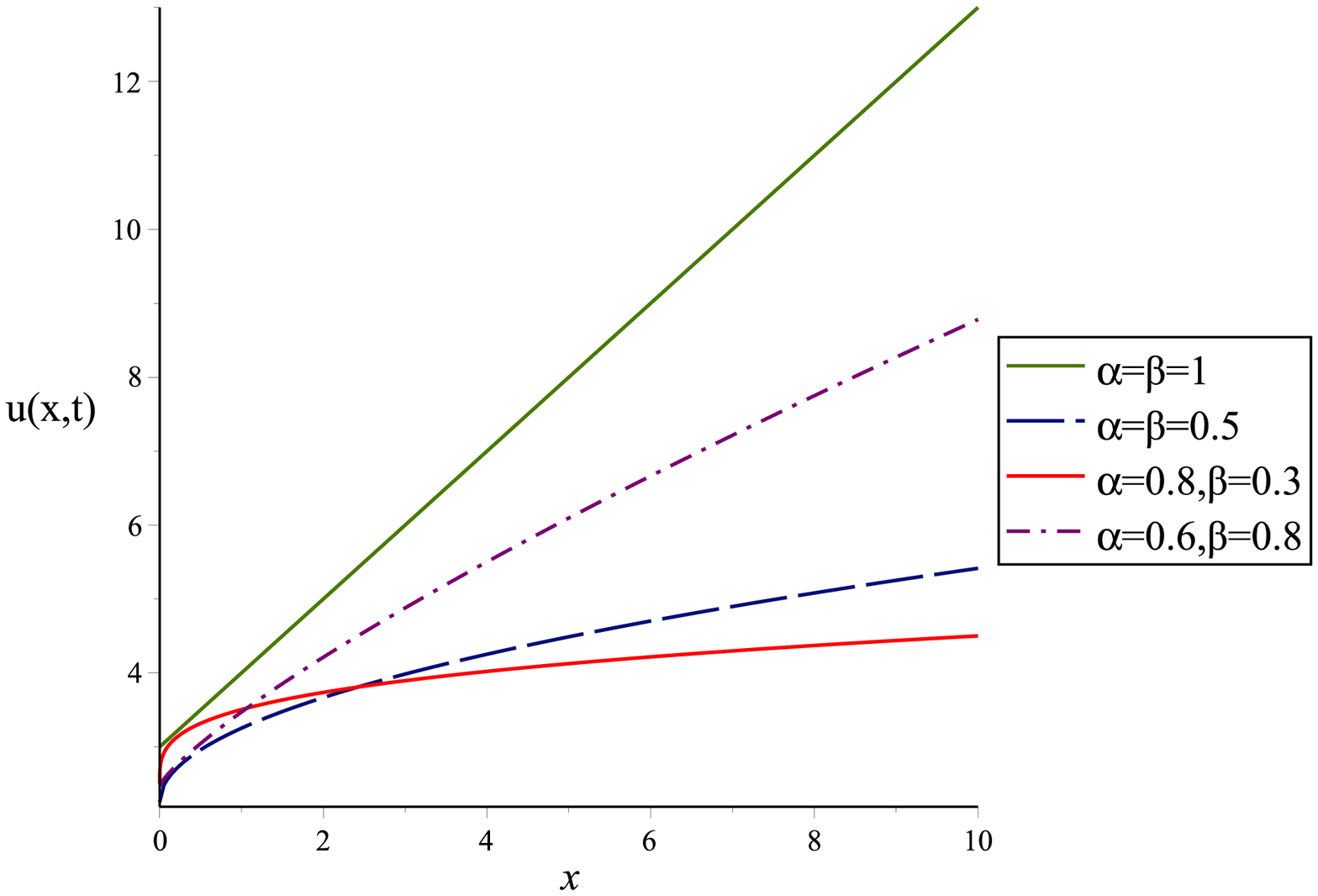}
 \footnotesize{Fig.(d) Graphical representation of the solution \eqref{RPP} for $k_0=k_1=1$, $t=2$, and different values of $\alpha$ and $\beta$.}
 \end{subfigure}
\end{center}
\end{figure}

\begin{rem}
We note that if $p(u)=1$ and $q(u)=-\dfrac{1}{2}u^{2}$, then the diffusion-convection equation \eqref{doce} reduces into time-space fractional Burgers equation
\begin{equation}
\dfrac{\partial^{\alpha}u}{\partial t^{\alpha}}=\dfrac{\partial^{\beta}}{\partial x^{\beta}}\left(\dfrac{\partial^{\beta}u}{\partial x^{\beta}}\right)+u \dfrac{\partial^{\beta}u}{\partial x^{\beta}},
\end{equation}
which admits a two-dimensional polynomial invariant subspace $\mathcal{W}_{2}=\mathfrak{L}\left\{1,x^{\beta}\right\}$.
\end{rem}
\subsection{Time-space fractional reaction-diffusion equation}
Consider the following time-space fractional reaction-diffusion equation
\begin{equation}\label{eqsr1}
\dfrac{\partial^{\alpha}u}{\partial t^{\alpha}}=\hat{G}[u]=p(u)\dfrac{\partial^{\beta+1}u}{\partial x^{\beta+1}}+q(u),\ \alpha,\beta\in(0,1],
\end{equation}
The above PDE with $\beta=1$ was discussed through the generalized differential transform method in \cite{st2}. We would like to point out that the differential operator $\hat{G}[u]$ admits no invariant subspace for arbitrary functions $p(u)$ and $q(u)$. Hence, we choose $p(u)=a_{n}u^{n}+a_{n-1}u^{n-1}+\dots+a_{1}u+a_{0}$,\ $q(u)=b_{n+1}u^{n+1}+b_{n}u^{n}+\dots+b_{1}u+b_0$, $n\in\mathbb{N}$ where $k$, $a_{n},b_{n+1},\dots,a_{1},b_{1},b_{0},a_{0}$ are non-zero arbitrary constants. Then, the time-space fractional reaction-diffusion equation \eqref{eqsr1} reduces to
\begin{eqnarray}\label{eqsr2}
\begin{aligned}
\dfrac{\partial^{\alpha}u}{\partial t^{\alpha}}=\hat{G}[u]=&\left[a_{n}u^{n}+a_{n-1}u^{n-1}+\dots+a_{1}u+a_{0}\right]\dfrac{\partial^{\beta+1}u}{\partial x^{\beta+1}}\\
&+\left(b_{n+1}u^{n+1}+b_{n}u^{n}+\dots+b_{1}u+b_0\right).
\end{aligned}
\end{eqnarray}
It is directly to check that the above equation \eqref{eqsr2} admits a one-dimensional invariant subspace $\mathcal{W}_{1}=\mathfrak{L}\left\{E_{\beta+1}(-kx^{\beta+1})\right\}$, because
$$\hat{G}[A_{1} E_{\beta+1}(-kx^{\beta+1})]=(-ka_0+b_1)A_{1} E_{\beta+1}(-kx^{\beta+1})\in \mathcal{W}_{1}$$ if $a_{i}k=b_{i+1}$, $i=1,2\dots n$, $n\in\mathbb{N}$ and $b_0=0$.\\
Following the above similar procedure, first, we consider $\alpha=\beta=1$. In this case, we have
\begin{equation}\label{eqsr5}
u(x,t)=k_{0}e^{(-ka_{0}+b_{1})t-x^{2}},
\end{equation}
where $k,k_{0}$, $b_{1}$ and $a_{0}$ are arbitrary constants.\\
Next, we assume $\alpha,\beta\in(0,1]$.
Thus, we obtain an exact solution of equation \eqref{eqsr2} as follows
\begin{equation}\label{eqsr6}
u(x,t)=k_{0}E_{\alpha}((-ka_{0}+b_{1})t^{\alpha})E_{\beta+1}(-kx^{\beta+1}),\ \alpha,\beta\in(0,1],
\end{equation}
where $k_{0}$, $k$, $b_{1}$ and $a_{0}$ are non-zero arbitrary constants. The above exact solution \eqref{eqsr6} for $k=b_1=1$, $a_{0}=-1$, $k_0=2$, $t=2$ and different values of $\alpha$ and $\beta$ is shown in Fig. (e). We observe that for $\alpha=\beta=1$, equation \eqref{eqsr6} is exactly same as \eqref{eqsr5}.
\begin{rem}
Let $a_{1}, b_{2}, b_1\in\mathbb{R}$ and $a_{n}=a_{n-1}=\dots=a_{0}=b_{n+1}=b_{n}=\dots=b_{1}=b_0=0$, that is, $p(u)=a_1 u$ and $q(u)=b_2 u^{2}+b_1 u$.\\
Then, equation \eqref{eqsr1} reduces into
\begin{eqnarray}\label{eqsr7}
\begin{aligned}
\dfrac{\partial^{\alpha}u}{\partial t^{\alpha}}=a_1 u\dfrac{\partial^{\beta+1}u}{\partial x^{\beta+1}}+b_2u^{2}+b_1u
\end{aligned}
\end{eqnarray}
which admits one-dimensional invariant subspace $\mathcal{W}_{2}=\mathfrak{L}\left\{E_{\beta+1}(-kx^{\beta+1})\right\}$ if $b_2=a_1k$. In this case, we obtain an exact solution
\begin{equation}\label{RPPP1}
  u(x,t)=k_{1} E_{\alpha}(b_{1}t^{\alpha})E_{\beta+1}(-kx^{\beta+1}),\ k,k_1,b_1\in\mathbb{R},\ \alpha,\beta\in(0,1],
\end{equation}
The above exact solution \eqref{RPPP1} for $k_1=1$, $b_1=-4$, $k=2$, $t=2$, and different values of $\alpha$ and $\beta$ is shown in Fig. (f).
\end{rem}
\begin{figure}[h!]
\begin{center}
\begin{subfigure}[]{.94\textwidth}
  \includegraphics[width=\textwidth]{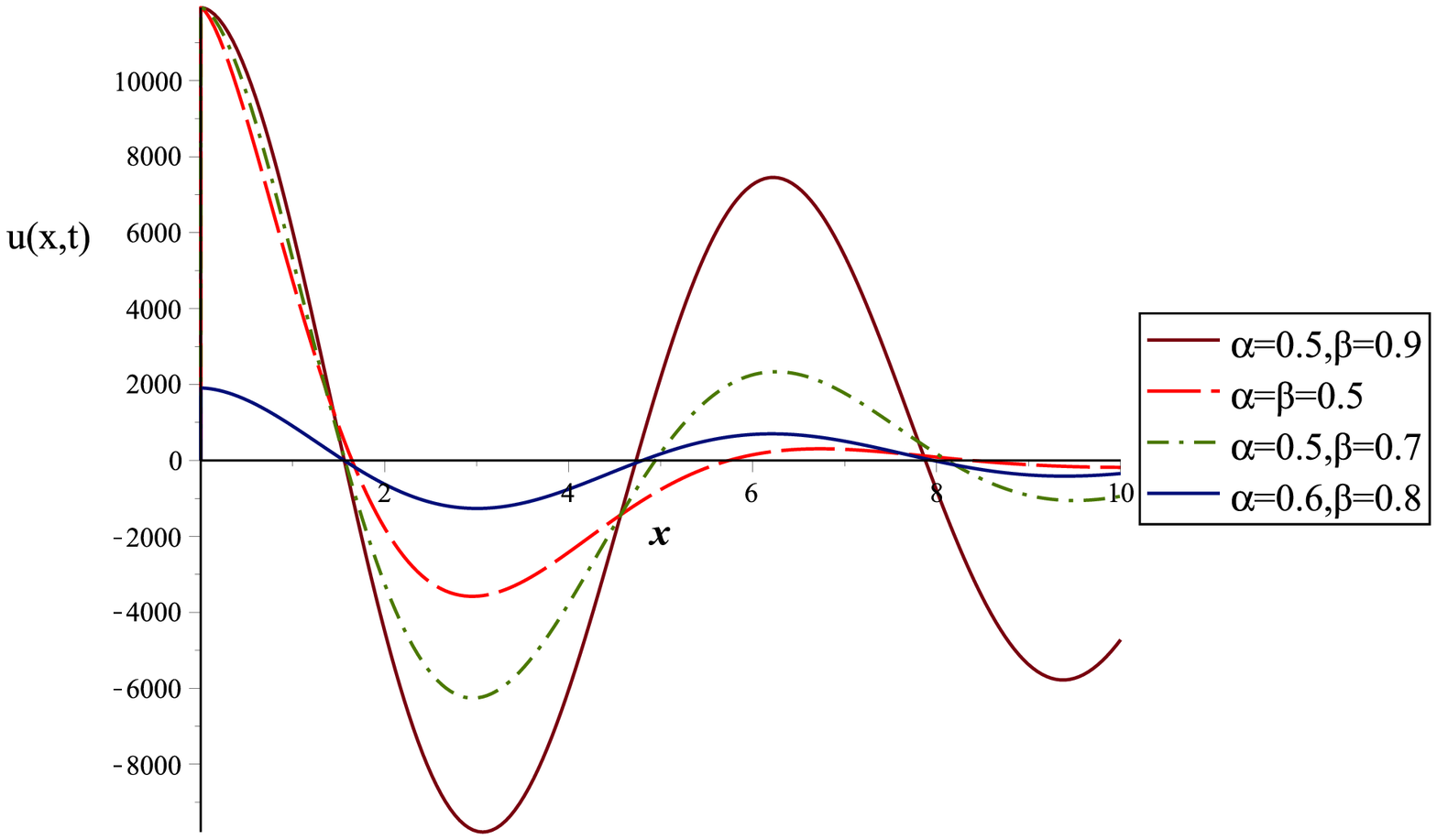}
\footnotesize{Fig.(e) Graphical representation of the solution \eqref{eqsr6} for $k=b_1=1$, $k_0=2$, $a_0=-1$, $t=2$ and different values of $\alpha$ and $\beta$.}
 \end{subfigure}\hspace{30pt}
 \begin{subfigure}[]{.94\textwidth}
  \includegraphics[width=\textwidth]{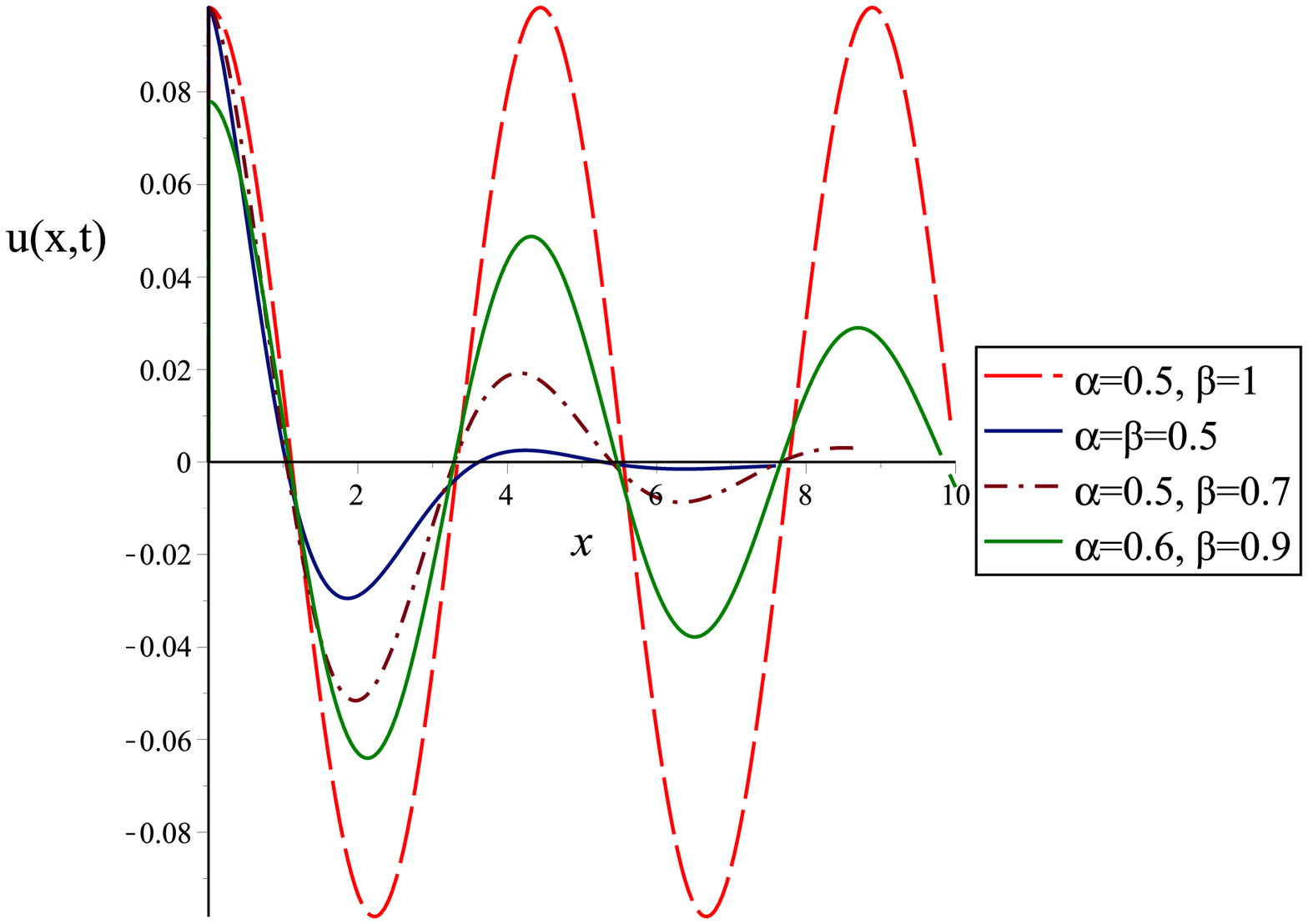}

 \footnotesize{Fig.(f) Graphical representation of the solution \eqref{RPPP1} for $k_1=1$, $b_1=-4$, $k=2$, $t=2$, and different values of $\alpha$ and $\beta$.}
 \end{subfigure}
\end{center}
\end{figure}
\begin{rem}
Let $a_{1}=1$, $b_{1}=-k\neq0$, $b_0, k\in\mathbb{R}$ and $a_{n}=a_{n-1}=\dots=a_{0}=b_{n+1}=b_{n}=\dots=b_{2}=0$, that is, $p(u)=u$ and $q(u)=-ku+b_0$.\\
Then, equation \eqref{eqsr2} reduces into
\begin{eqnarray}\label{eqsr8}
\begin{aligned}
\dfrac{\partial^{\alpha}u}{\partial t^{\alpha}}=u\dfrac{\partial^{\beta+1}u}{\partial x^{\beta+1}}-ku+b_0
\end{aligned}
\end{eqnarray}
which admits the following distinct invariant subspaces
\begin{enumerate}
\item[(i)]\ $\mathcal{W}_{2}=\mathfrak{L}\left\{1,x^{\beta}\right\}$,
\item[(ii)]\ $\mathcal{W}_{2}=\mathfrak{L}\left\{1,x^{\beta+1}\right\}$,
 \item[(iii)]\ $\mathcal{W}_{3}=\mathfrak{L}\left\{1,x^{\beta},x^{\beta+1}\right\}$.
 \end{enumerate}
  Now, we consider the invariant subspace $\mathcal{W}_{2}=\mathfrak{L}\left\{1,x^{\beta}\right\}$. In this case, we yield an exact solution of equation \eqref{eqsr8} as follows
 \begin{equation}\label{RPPP2}
 u(x,t)=(k_{1}+k_{2}x^{\beta})E_{\alpha}(-kt^{\alpha})+b_0t^{\alpha} E_{\alpha,\alpha+1}(-kt^\alpha),\ k,k_1,k_2,b_0\in\mathbb{R},\ \alpha,\beta\in(0,1].
 \end{equation}
 The above exact solution \eqref{RPPP2} for $k_1=k_2=1$, $b_0=0$, $k=2$, $x=2$, and different values of $\alpha$ and $\beta$ is shown in Fig. (g).
\end{rem}
\begin{rem}
Let $a_{0}=c\in\mathbb{R}$, $b_{1}=-k$ and $a_{n}=a_{n-1}=\dots=a_{1}=b_{n+1}=b_{n}=\dots=b_{2}=b_0=0$, that is, $p(u)=c$ and $q(u)=-ku$.\\
Then, equation \eqref{eqsr2} reduces into
\begin{eqnarray}\label{eqsr9}
\begin{aligned}
\dfrac{\partial^{\alpha}u}{\partial t^{\alpha}}=c\dfrac{\partial^{\beta+1}u}{\partial x^{\beta+1}}-ku
\end{aligned}
\end{eqnarray}
which admits distinct invariant subspaces, that is,
\begin{enumerate}
\item[(i)]\ $\mathcal{W}_{2}=\mathfrak{L}\left\{1,x^{\beta}\right\},$
  \item[(ii)]\ $\mathcal{W}_{3}=\mathfrak{L}\left\{1,x^{\beta},x^{\beta+1}\right\},$
\item[(iii)]\ $\mathcal{W}_{n}=\mathfrak{L}\left\{E_{\beta+1}(k_{1}x^{\beta+1}),\dots,E_{\beta+1}(k_{n}x^{\beta+1})\right\}$, $n\in\mathbb{N}$,
\item[(iv)]\ $\mathcal{W}_{n+1}=\mathfrak{L}\left\{1,E_{\beta+1}(k_{1}x^{\beta+1}),\dots,E_{\beta+1}(k_{n}x^{\beta+1})\right\}$, $n\in\mathbb{N},$
\item[(v)]\ $\mathcal{W}_{n+2}=\mathfrak{L}\left\{1,x^{\beta+1},E_{\beta+}(k_{1}x^{\beta+1}),\dots,E_{\beta+1}(k_{n}x^{\beta+1})\right\}$, $n\in\mathbb{N}$,
\item[(vi)]\ $\mathcal{W}_{n+3}=\mathfrak{L}\left\{1,x^{\beta},x^{\beta+1},E_{\beta+1}(k_{1}x^{\beta+1}),\dots,E_{\beta+1}(k_{n}x^{\beta+1})\right\}$, $n\in\mathbb{N}$,
\end{enumerate}
 where $k_{i}\in\mathbb{R}$, $i=1,2,\dots,n$. Now, we consider the more general invariant subspace $\mathcal{W}_{n+3}=\mathfrak{L}\left\{1,x^{\beta},x^{\beta+1},E_{\beta+1}(k_{1}x^{\beta+1}),\dots,E_{\beta+1}(k_{n}x^{\beta+1})\right\}$, which suggests that equation \eqref{eqsr9} possesses the more general exact solution
\begin{eqnarray}
\begin{aligned}\label{sr7}
  u(x,t) =& \left(\lambda_{1}+\lambda_{2}x^{\beta}+\lambda_{3}x^{\beta+1}\right)E_{\alpha}(-kt^{\alpha})+\sum\limits^{n}_{r=1}\lambda_{r+3}E_{\alpha}((k_{r}c-k)t^{\alpha})E_{\beta+1}(k_{r}x^{\beta+1})\\
   +&c \lambda_{3}\Gamma(\beta+2)\int\limits^{t}_{0}E_{\alpha}(-k(t-\tau)^{\alpha}) (\tau)^{\alpha-1}E_{\alpha,\alpha}(-k\tau^{\alpha})d\tau,
  \end{aligned}
\end{eqnarray}
where $k$, $c$, $k_{i}$ $(i=1,2,\dots,n)$ and $\lambda_{s}$ (s=1,2,\dots,n+3), are arbitrary constants.
\end{rem}
\begin{rem}
Let $a_{0}=c$, $c\in\mathbb{R}$, and $a_{n}=a_{n-1}=\dots=a_{1}=b_{n+1}=b_{n}=\dots=b_{1}=b_0=0$, that is, $p(u)=c$ and $q(u)=0$.\\
Then, fractional reaction-diffusion equation \eqref{eqsr2} reduces into the fractional linear sub-diffusion equation
\begin{eqnarray}\label{eqsr10}
\begin{aligned}
\dfrac{\partial^{\alpha}u}{\partial t^{\alpha}}=c\dfrac{\partial^{\beta+1}u}{\partial x^{\beta+1}}
\end{aligned}
\end{eqnarray}
which admits the following distinct invariant subspaces:
\begin{enumerate}
\item[(i)]\ $\mathcal{W}_{2}=\mathfrak{L}\left\{1,x^{\beta}\right\},$
  \item[(ii)]\ $\mathcal{W}_{3}=\mathfrak{L}\left\{1,x^{\beta},x^{\beta+1}\right\},$
\item[(iii)]\ $\mathcal{W}_{n}=\mathfrak{L}\left\{E_{\beta+1}(k_{1}x^{\beta+1}),\dots,E_{\beta+1}(k_{n}x^{\beta+1})\right\}$, $n\in\mathbb{N}$,
\item[(iv)]\ $\mathcal{W}_{n+1}=\mathfrak{L}\left\{1,E_{\beta+1}(k_{1}x^{\beta+1}),\dots,E_{\beta+1}(k_{n}x^{\beta+1})\right\}$, $n\in\mathbb{N},$
\item[(v)]\ $\mathcal{W}_{n+2}=\mathfrak{L}\left\{1,x^{\beta+1},E_{\beta+}(k_{1}x^{\beta+1}),\dots,E_{\beta+1}(k_{n}x^{\beta+1})\right\}$, $n\in\mathbb{N}$,
\item[(vi)]\ $\mathcal{W}_{n+3}=\mathfrak{L}\left\{1,x^{\beta},x^{\beta+1},E_{\beta+1}(k_{1}x^{\beta+1}),\dots,E_{\beta+1}(k_{n}x^{\beta+1})\right\}$, $n\in\mathbb{N}$,
\end{enumerate}
where $k_{i}\in\mathbb{R}$, $i=1,2,\dots,n$. Here, we consider the more general invariant subspace (vi), which possesses the more general exact solution of fractional sub-diffusion equation \eqref{eqsr10} reads
\begin{eqnarray}
\begin{aligned}\label{sr8}
  u(x,t) =& \left(\lambda_{1}+\lambda_{2}x^{\beta}+\lambda_{3}x^{\beta+1}\right)+\sum\limits^{n}_{r=1}\lambda_{r+3}E_{\alpha}(k_{r}ct^{\alpha})E_{\beta+1}(k_{r}x^{\beta+1})\\
   +&c \lambda_{3}\dfrac{\Gamma(\beta+2)}{\Gamma(\alpha+1)}t^{\alpha},\ \alpha,\beta\in(0,1],
  \end{aligned}
\end{eqnarray}
where $c$, $k_{i}$ $(i=1,2,\dots,n)$ and $\lambda_{s}$ (s=1,2,\dots,n+3), are arbitrary constants. The above exact solution \eqref{sr8} for $\lambda_1=\lambda_2=\lambda_3=\lambda_4=k_1=c=n=1$, $t=2$, and different values of $\alpha$ and $\beta$ is shown in Fig. (h). Observe that for $k=0$, equation \eqref{sr7} is exactly same as equation \eqref{sr8}.

\begin{figure}[h!]
\begin{center}
\begin{subfigure}[]{0.88\textwidth}
  \includegraphics[width=\textwidth]{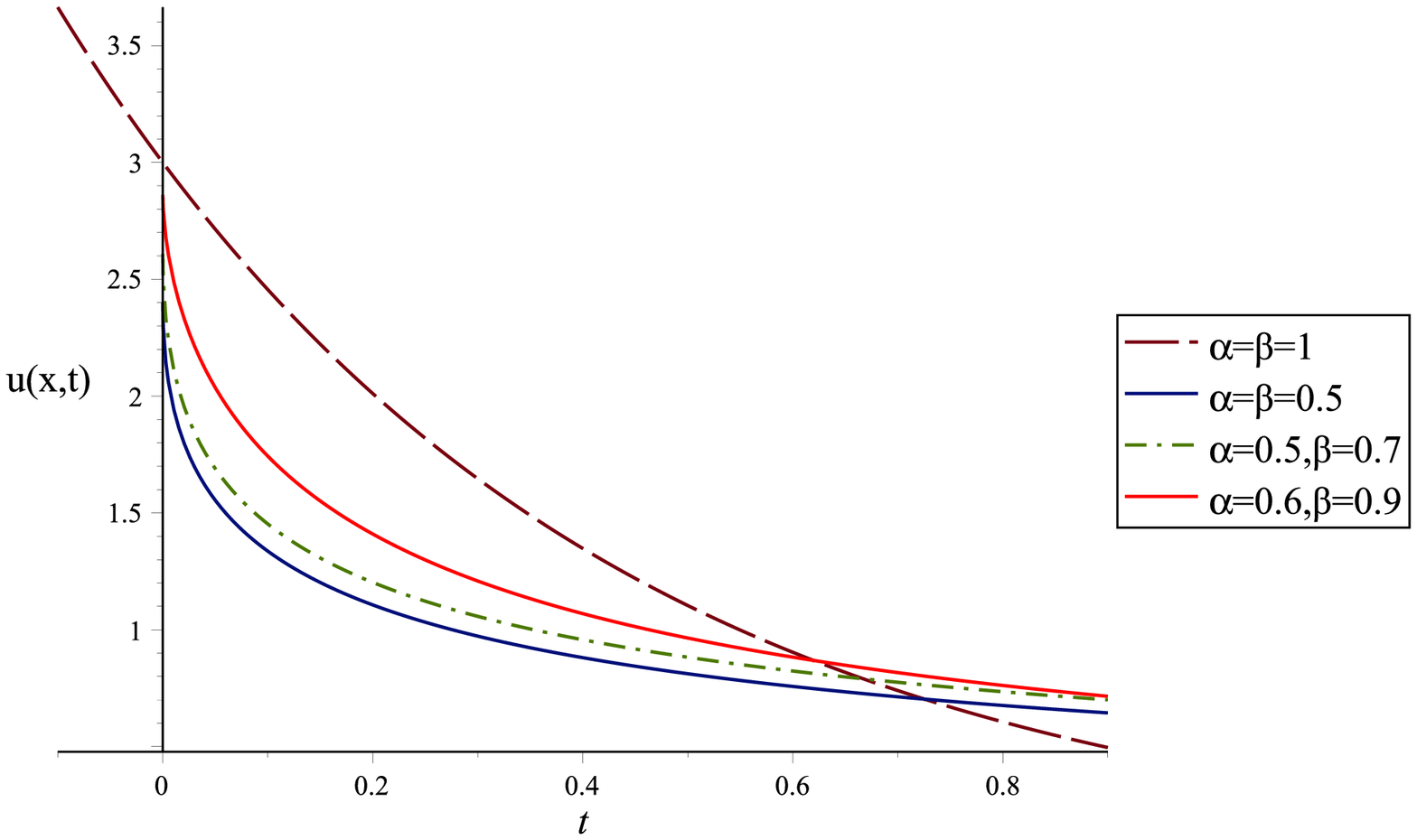}
\footnotesize{Fig.(g) Graphical representation of the solution \eqref{RPPP2} for $k_1=k_2=1$, $b_0=0$, $k=2$, $x=2$, and different values of $\alpha$ and $\beta$.}
 \end{subfigure}\hspace{30pt}
 \begin{subfigure}[]{0.88\textwidth}
  \includegraphics[width=\textwidth]{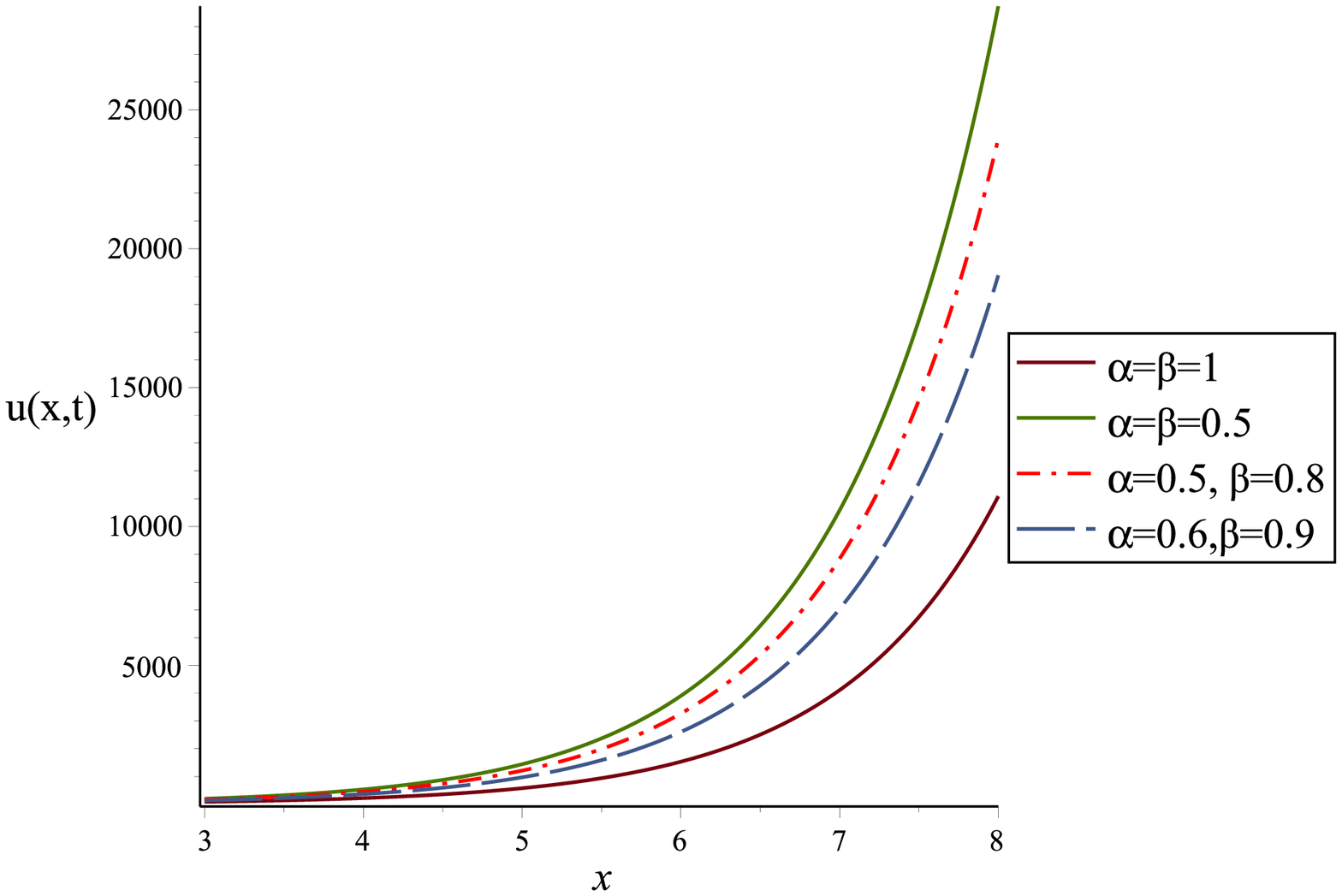}
 \footnotesize{Fig.(h) Graphical representation of the solution \eqref{sr8} for $\lambda_1=\lambda_2=\lambda_3=\lambda_4=k_1=c=n=1$, $t=2$, and different values of $\alpha$ and $\beta$.}
 \end{subfigure}
\end{center}
\end{figure}

\end{rem}
\subsection{Time-space fractional diffusion equation with source term}
Consider a time-space fractional nonlinear diffusion equation with source term
\begin{equation}\label{DS1}
\dfrac{\partial^{\alpha}u}{\partial t^{\alpha}}=\hat{G}[u]=\left(\dfrac{\partial p}{\partial u}\right)\left(\dfrac{\partial^{\beta}u}{\partial x^{\beta}}\right)^{2}+p(u)\dfrac{\partial^{\beta}}{\partial x^{\beta}}\left(\dfrac{\partial^{\beta}u}{\partial x^{\beta}}\right)+q(u),\ \alpha,\beta\in(0,1].
\end{equation}
We would like to point out that the operator $\hat{G}[u]$ admits no invariant subspace for arbitrary functions $p(u)$ and q(u). Hence, we choose $p(u)=a_{n}u^{n}+a_{n-1}u^{n-1}+\dots+a_{1}u+a_{0}$ and $q(u)=b_{n+1}u^{n+1}+b_{n}u^{n}+\dots+b_{1}u$, where $a_{n},b_{n+1},\dots,a_{1},b_{1},a_{0}$ are arbitrary constants. Then, the equation \eqref{DS1} describes the time-space diffusion equation with source term
\begin{eqnarray}\label{DS2}
\begin{aligned}
\dfrac{\partial^{\alpha}u}{\partial t^{\alpha}}=&\left(na_{n}u^{n-1}+(n-1)a_{n-1}u^{n-2}+\dots+a_{1}\right)\left(\dfrac{\partial^{\beta}u}{\partial x^{\beta}}\right)^{2}\\
&+\left(a_{n}u^{n}+a_{n-1}u^{n-1}+\dots+a_{1}u+a_{0}\right)\dfrac{\partial^{\beta}}{\partial x^{\beta}}\left(\dfrac{\partial^{\beta}u}{\partial x^{\beta}}\right)\\
&+b_{n+1}u^{n+1}+b_{n}u^{n}+\dots+b_{1}u.
\end{aligned}
\end{eqnarray}
It is easy check that the above equation \eqref{DS2} admits a one-dimensional invariant subspace $\mathcal{W}_{1}=\mathfrak{L}\left\{E_{\beta}(kx^{\beta})\right\}$ if $(i+1)a_{i}k^{2}=-b_{i+1}$, $i=1,2\dots,n$, $n\in\mathbb{N}$.
Thus, for $\alpha=\beta=1$, we have
\begin{equation}\label{DS4}
u(x,t)=k_{0}e^{(a_{0}k^{2}+b_{1})t+kx},\ b_{1},k_{0},k,a_{0}\in\mathbb{R}.
\end{equation}
Now, we consider $\alpha,\beta\in(0,1]$. In this case, we obtain an exact solution of equation \eqref{DS2} reads
\begin{equation}\label{DS6}
u(x,t)=k_{0}E_{\alpha}((a_{0}k^{2}+b_{1})t^{\alpha})E_{\beta}(kx^{\beta}),\ b_{1},k_{0},k,a_{0}\in\mathbb{R}.
\end{equation}
Note that if $\alpha=\beta=1$,  then equation \eqref{DS6} is exactly same as equation \eqref{DS4}. The above exact solution \eqref{DS6} for $k_0=b_1=1$, $a_{0}=0$, $k=2$, $t=2$, and different values of $\alpha$ and $\beta$ is shown in Fig. (i).
\begin{rem}
Let $p(u)=a_{0}$ and $q(u)=b_{1}u+b_{0}$, $a_{0}$, $b_{0}$ and $b_{1}$ are arbitrary constants.
Then, equation \eqref{DS1} reduces into
\begin{equation}\label{DS7}
\dfrac{\partial^{\alpha}u}{\partial t^{\alpha}}=a_{0}\dfrac{\partial^{\beta}}{\partial x^{\beta}}\left(\dfrac{\partial^{\beta}u}{\partial x^{\beta}}\right)+b_{1}u+b_{0}
\end{equation}
which admits distinct invariant subspaces, that is,
\begin{enumerate}
\item[(i)]\ $\mathcal{W}_{2}=\mathfrak{L}\left\{1,x^{\beta}\right\}$,
\item[(ii)]\ $\mathcal{W}_{n+1}=\mathfrak{L}\left\{1,E_{\beta}(k_{1}x^{\beta}),E_{\beta}(k_{2}x^{\beta}),\dots,E_{\beta}(k_{n}x^{\beta})\right\}$,
\item[(iii)]\ $\mathcal{W}_{n+2}=\mathfrak{L}\left\{1,x^{\beta},E_{\beta}(k_{1}x^{\beta}),E_{\beta}(k_{2}x^{\beta}),\dots,E_{\beta}(k_{n}x^{\beta})\right\},\ k_{i}\in\mathbb{R},\ i=1,2,\dots,n.$
\end{enumerate}
Let us consider the more general invariant subspace (iii), that is,\\
$\mathcal{W}_{n+2}=\mathfrak{L}\left\{1,x^{\beta},E_{\beta}(k_{1}x^{\beta}),E_{\beta}(k_{2}x^{\beta}),\dots,E_{\beta}(k_{n}x^{\beta})\right\}.$
Thus, we obtain an exact solution of \eqref{DS7} reads
\begin{equation}\label{DS10}
u(x,t)=(c_{1}+c_{2}x^{\beta})E_{\alpha}(b_{1}t^{\alpha})+b_{0}t^{\alpha}E_{\alpha,\alpha+1}(b_{1}t^{\alpha})+\sum\limits^{n}_{r=1}c_{r+2}E_{\alpha}((b_{1}+a_{0}k_{r}^{2})t^{\alpha})E_{\beta}(k_{r}x^{\beta}),
\end{equation}
where $a_{0}$, $b_{0}$, $b_{1}$, $k_{i} (i=1,2,\dots,n)$ and $c_{s}$ $(s=1,2,\dots,n+2)$ are arbitrary constants. The above exact solution \eqref{DS10} for  $c_1=c_2=c_3=b_0=b_1=a_0=k_1=n=1$, $x=2$, and different values of $\alpha$ and $\beta$ is shown in Fig. (j).
\end{rem}
\begin{figure}[h!]
\begin{center}
\begin{subfigure}[]{0.88\textwidth}
  \includegraphics[width=\textwidth]{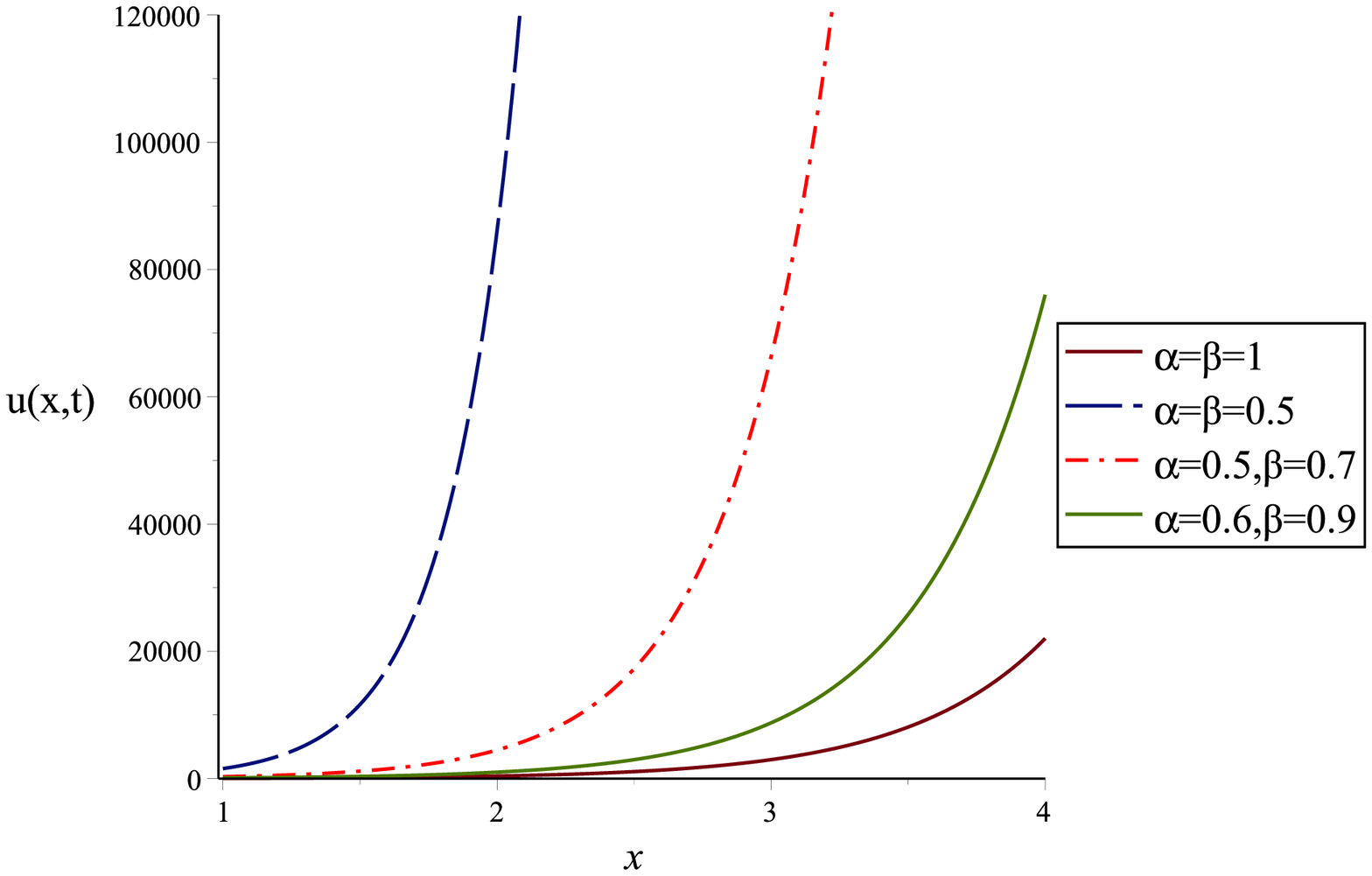}
\footnotesize{Fig.(i) Graphical representation of the solution \eqref{DS6} for $k_0=b_1=1$, $a_{0}=0$, $k=2$, $t=2$, and different values of $\alpha$ and $\beta$.}
 \end{subfigure}\hspace{30pt}
 \begin{subfigure}[]{0.88\textwidth}
  \includegraphics[width=\textwidth]{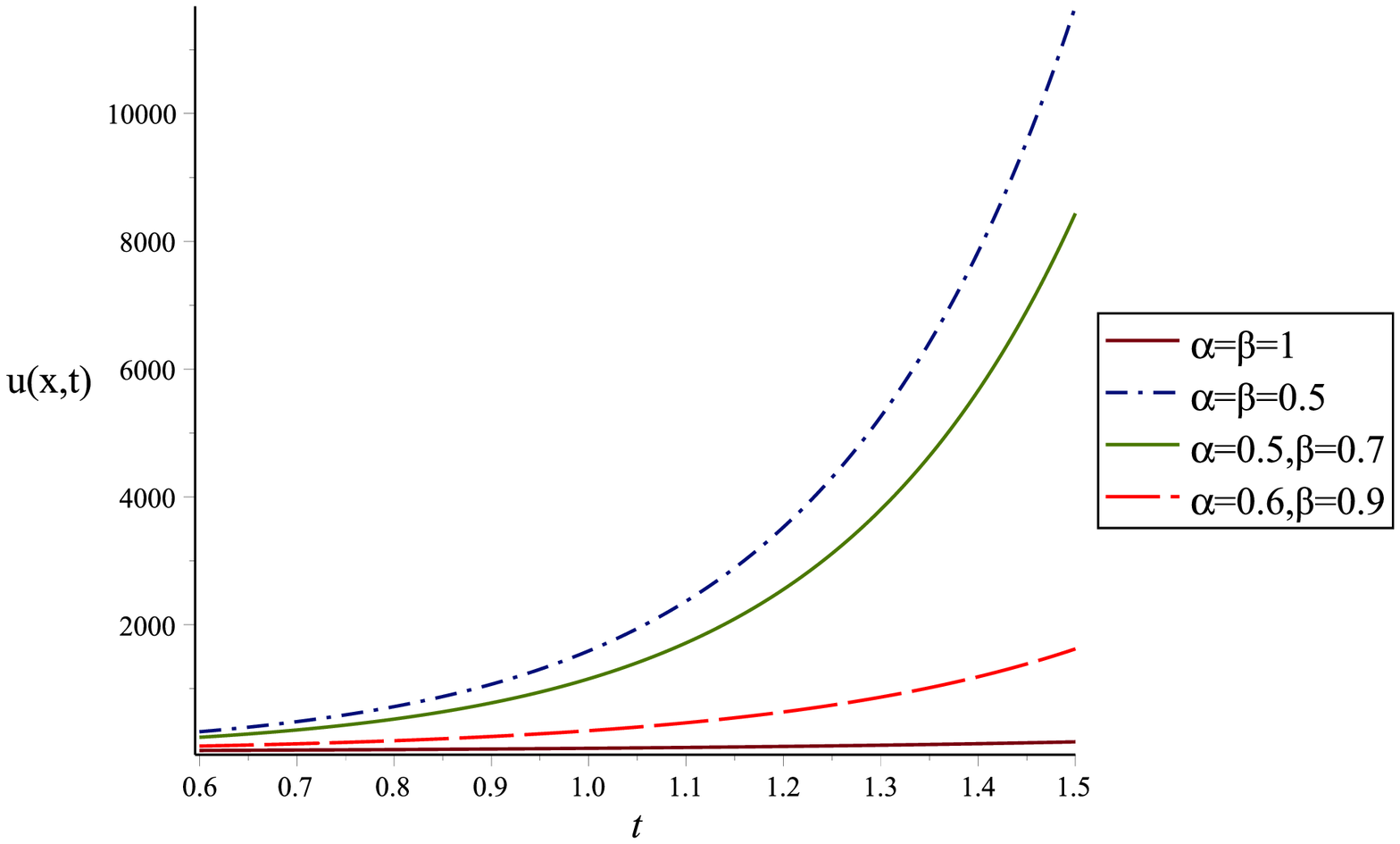}
 \footnotesize{Fig.(j) Graphical representation of the solution $u(x,t)$ of \eqref{DS10} for $c_1=c_2=c_3=b_0=b_1=a_0=k_1=n=1$, $x=2$, and different values of $\alpha$ and $\beta$.}
 \end{subfigure}
\end{center}
\end{figure}
\begin{rem}
Let $p(u)=a_{0}$ and $q(u)=b_{1}u$, $a_{0}$ and $b_{1}$ are arbitrary constants.
Then, equation \eqref{DS1} reduces into
\begin{equation}\label{DS8}
\dfrac{\partial^{\alpha}u}{\partial t^{\alpha}}=a_{0}\dfrac{\partial^{\beta}}{\partial x^{\beta}}\left(\dfrac{\partial^{\beta}u}{\partial x^{\beta}}\right)+b_{1}u
\end{equation}
which admits the following distinct invariant subspaces:\\
$(i)\ \mathcal{W}_{2}=\mathfrak{L}\left\{1,x^{\beta}\right\}$,\\
$(ii)\ \mathcal{W}_{n}=\mathfrak{L}\left\{E_{\beta}(k_{1}x^{\beta}),E_{\beta}(k_{2}x^{\beta}),\dots,E_{\beta}(k_{n}x^{\beta})\right\},$\\
$(iii)\ \mathcal{W}_{n+1}=\mathfrak{L}\left\{1,E_{\beta}(k_{1}x^{\beta}),E_{\beta}(k_{2}x^{\beta}),\dots,E_{\beta}(k_{n}x^{\beta})\right\},$\\
$(iv)\ \mathcal{W}_{n+2}=\mathfrak{L}\left\{1,x^{\beta},E_{\beta}(k_{1}x^{\beta}),E_{\beta}(k_{2}x^{\beta}),\dots,E_{\beta}(k_{n}x^{\beta})\right\}$, $n\in\mathbb{N}$, $k_{i}\in\mathbb{R}$,\ i=1,2,\dots,n.\\
Let us consider the more general invariant subspace (iv), that is,\\
$\mathcal{W}_{n+2}=\mathfrak{L}\left\{1,x^{\beta},E_{\beta}(k_{1}x^{\beta}),E_{\beta}(k_{2}x^{\beta}),\dots,E_{\beta}(k_{n}x^{\beta})\right\}.$
Thus, we obtain the more general exact solution of \eqref{DS8} reads
\begin{equation}\label{DS11}
u(x,t)=(c_{1}+c_{2}x^{\beta})E_{\alpha}(b_{1}t^{\alpha})+\sum\limits^{n}_{r=1}c_{r+2}E_{\alpha}((b_{1}+a_{0}k_{r}^{2})t^{\alpha})E_{\beta}(k_{r}x^{\beta}),
\end{equation}
where $a_{0}$, $b_{1}$, $k_{i} (i=1,2,\dots,n)$ and $c_{s}$ $(s=1,2,\dots,n+2)$ are arbitrary constants. Note that for $b_{0}=0$, equation \eqref{DS10} is exactly same as equation \eqref{DS11}.
\end{rem}
\begin{rem}
Let $p(u)=a_{1}u+a_{0}$ and $q(u)=b_{1}u+b_{0}$, $a_{0}$, $a_{1}$, $b_{0}$ and $b_{1}$ are arbitrary constants.
Then, equation \eqref{DS1} reduces into
\begin{equation}\label{DS9}
\dfrac{\partial^{\alpha}u}{\partial t^{\alpha}}=a_{1}\left(\dfrac{\partial^{\beta}u}{\partial x^{\beta}}\right)^{2}+(a_{1}u+a_{0})\dfrac{\partial^{\beta}}{\partial x^{\beta}}\left(\dfrac{\partial^{\beta}u}{\partial x^{\beta}}\right)+b_{1}u+b_{0}
\end{equation}
which admits two-dimensional invariant subspace $\mathcal{W}_{2}=\mathfrak{L}\left\{1,x^{\beta}\right\}$.
\end{rem}
\subsection{Two-coupled system of time-space fractional diffusion equation}
Consider the following two-coupled system of time-space fractional diffusion equation
\begin{eqnarray}
\begin{aligned}\label{cc1}
&\dfrac{\partial^{\alpha}u_{1}}{\partial t^{\alpha}}= \dfrac{\partial^{\beta}}{\partial x^{\beta}}\left(\dfrac{\partial^{\beta}u_{1}}{\partial x^{\beta}}\right)+\mu u_{2}\dfrac{\partial^{\beta}}{\partial x^{\beta}}\left(\dfrac{\partial^{\beta}u_{1}}{\partial x^{\beta}}\right)+(\mu+\rho)\left(\dfrac{\partial^{\beta}u_{1}}{\partial x^{\beta}}\right)\left(\dfrac{\partial^{\beta}u_{2}}{\partial x^{\beta}}\right)+\rho u_{1}\dfrac{\partial^{\beta}}{\partial x^{\beta}}\left(\dfrac{\partial^{\beta}u_{2}}{\partial x^{\beta}}\right),\\
&\dfrac{\partial^{\alpha}u_{2}}{\partial t^{\alpha}}=\dfrac{\partial^{\beta}}{\partial x^{\beta}}\left(\dfrac{\partial^{\beta}u_{2}}{\partial x^{\beta}}\right)+\lambda\dfrac{\partial^{\beta}}{\partial x^{\beta}}\left(\dfrac{\partial^{\beta}u_{1}}{\partial x^{\beta}}\right)+\gamma u_{1}+\delta u_{2},\ \alpha,\beta\in(0,1].
\end{aligned}
\end{eqnarray}
It is easy to find that coupled system \eqref{cc1} admits the following two-dimensional distinct invariant subspaces:\\
$(i)\ \mathcal{W}_{2}^{1}\times \mathcal{W}_{2}^{2}=\mathfrak{L}\left\{E_{\beta}(kx^{\beta}),E_{\beta}(-kx^{\beta})\right\}\times\mathfrak{L}\left\{E_{\beta}(kx^{\beta}),E_{\beta}(-kx^{\beta})\right\}$ if $\mu=-\rho$,\\
$(ii)\ \mathcal{W}_{2}^{1}\times \mathcal{W}_{2}^{2}=\mathfrak{L}\left\{1,x^{\beta}\right\}\times\mathfrak{L}\left\{1,x^{\beta}\right\}$, \\
because
\begin{eqnarray*}
&&\hat{G}_{1}\left[A_{1}E_{\beta}(kx^{\beta})+A_{2}E_{\beta}(-kx^{\beta}), A_{3}E_{\beta}(kx^{\beta})+A_{4}E_{\beta}(-kx^{\beta})\right]\\
&&=k^{2}\left[A_{1}E_{\beta}(kx^{\beta})+A_{2}E_{\beta}(-kx^{\beta})\right]\in W_{2}^{1},\ \text{if}\ \mu=-\rho,\\
&&\hat{G}_{2}\left[A_{1}E_{\beta}(kx^{\beta})+A_{2}E_{\beta}(-kx^{\beta}), A_{3}E_{\beta}(kx^{\beta})+A_{4}E_{\beta}(-kx^{\beta})\right]\\
&&=\left[(k^{2}+\delta)A_{3}+(\lambda k^{2}+\gamma)A_{1}\right]E_{\beta}(kx^{\beta})+\left[(k^{2}+\delta)A_{4}+(\lambda k^{2}+\gamma)A_{2}\right]E_{\beta}(-kx^{\beta})\in W_{2}^{2},
\end{eqnarray*}
and
\begin{eqnarray*}
\hat{G}_{1}\left[A_{1}+A_{2}x^{\beta}, A_{3}+A_{4}x^{\beta}\right]&=&(\rho+\mu)\left(\Gamma(\beta+1)\right)^{2}A_{2}A_{4}\in W_{2}^{1},\\
\hat{G}_{2}\left[A_{1}+A_{2}x^{\beta}, A_{3}+A_{4}x^{\beta}\right]&=&\left(\gamma A_{1}+\delta A_{3}\right)+\left(\gamma A_{2}+\delta A_{4}\right)x^{\beta}\in W_{2}^{2}.
\end{eqnarray*}
\textbf{Case 1.}\\
 First, we consider the invariant subspace $$\mathcal{W}_{2}^{1}\times \mathcal{W}_{2}^{2}=\mathfrak{L}\left\{E_{\beta}(kx^{\beta}),E_{\beta}(-kx^{\beta})\right\}\times\mathfrak{L}\left\{E_{\beta}(kx^{\beta}),E_{\beta}(-kx^{\beta})\right\}$$ with $\mu=-\rho$, which suggests that equation \eqref{cc1} admits an exact solution in the form
\begin{eqnarray}\label{cc2}
\begin{aligned}
&u_{1}(x,t)=A_{1}(t)E_{\beta}(kx^{\beta})+A_{2}(t)E_{\beta}(-kx^{\beta}),\\
&u_{2}(x,t)=A_{3}(t)E_{\beta}(kx^{\beta})+A_{4}(t)E_{\beta}(-kx^{\beta}),
\end{aligned}
\end{eqnarray}
where $A_{1}(t)$, $A_{2}(t)$, $A_{3}(t)$ and $A_{4}(t)$ are unknown functions to be determined. Substituting \eqref{cc2} in \eqref{cc1}, we get
\begin{eqnarray}
\begin{aligned}\label{cc3}
&\dfrac{d^{\alpha}A_{1}}{dt^{\alpha}}=k^{2}A_{1},\\
&\dfrac{d^{\alpha}A_{2}}{dt^{\alpha}}=k^{2}A_{2},\\
&\dfrac{d^{\alpha}A_{3}}{dt^{\alpha}}=(k^{2}+\delta)A_{3}+(\lambda k^{2}+\gamma)A_{1},\\
&\dfrac{d^{\alpha}A_{4}}{dt^{\alpha}}=(k^{2}+\delta)A_{4}+(\lambda k^{2}+\gamma)A_{2}.
\end{aligned}
\end{eqnarray}
To obtain a nonzero solution, we assume that $A_{i}(0)=a_{i}\neq0$, $i=1,2,3,4$.
Let us first consider $\alpha=\beta=1$. In this case, we have
\begin{eqnarray}\label{cc7}
\begin{aligned}
u_{1}(x,t)=&e^{k^{2}t}\left(a_{1}e^{kx}+a_{2}e^{-kx}\right),\\
u_{2}(x,t)=&\left[a_{3}e^{\delta t}+(\lambda k^{2}+\gamma)a_{1}\left(\dfrac{e^{\delta t}-1}{\delta}\right)\right]e^{k(kt+x)}\\
+&\left[a_{4}e^{\delta t}+(\lambda k^{2}+\gamma)a_{2}\left(\dfrac{e^{\delta t}-1}{\delta}\right)\right]e^{k(kt-x)},\ \delta\neq0,
\end{aligned}
\end{eqnarray}
where $a_{i},k,\lambda,\gamma,\delta\in\mathbb{R}$ $(i=1,2,3,4)$.
Next, we consider $\alpha\in(0,1]$. Applying Laplace transformation technique to linear system \eqref{cc3}, we obtain
\begin{eqnarray*}
\begin{aligned}
A_{1}(t)=&a_{1}E_{\alpha}(k^{2}t^{\alpha}),\\
A_{2}(t)=&a_{2}E_{\alpha}(k^{2}t^{\alpha}),\\
A_{3}(t)=&a_{3}E_{\alpha}\left((k^{2}+\delta)t^{\alpha}\right)+(\lambda k^{2}+\gamma)a_{1}\int\limits^{t}_{0}\tau^{\alpha-1}E_{\alpha,\alpha}\left((k^{2}+\delta)\tau^{\alpha}\right)E_{\alpha}(k^{2}(t-\tau)^{\alpha})d\tau,\\
A_{4}(t)=&a_{4}E_{\alpha}\left((k^{2}+\delta)t^{\alpha}\right)+(\lambda k^{2}+\gamma)a_{2}\int\limits^{t}_{0}\tau^{\alpha-1}E_{\alpha,\alpha}((k^{2}+\delta)\tau^{\alpha})E_{\alpha}(k^{2}(t-\tau)^{\alpha})d\tau.
\end{aligned}
\end{eqnarray*}
 In this case, we obtain an exact solution of the time-space fractional coupled diffusion system \eqref{cc1} associated with the invariant subspace $ \mathcal{W}_{2}^{1}\times \mathcal{W}_{2}^{2}=\mathfrak{L}\left\{E_{\beta}(kx^{\beta}),E_{\beta}(-kx^{\beta})\right\}\times\mathfrak{L}\left\{E_{\beta}(kx^{\beta}),E_{\beta}(-kx^{\beta})\right\}$ if $\rho=-\mu$, as follows
\begin{eqnarray}
\begin{aligned}\label{cc8}
u_{1}(x,t)=&\left[a_{1}E_{\beta}(kx^{\beta})+a_{2}E_{\beta}(-kx^{\beta})\right]E_{\alpha}(k^{2}t^{\alpha}),\\
u_{2}(x,t)=&\left[a_{3}E_{\beta}(kx^{\beta})+a_{4}E_{\beta}(-kx^{\beta})\right]E_{\alpha}((k^{2}+\delta)t^{\alpha})+(\lambda k^{2}+\gamma)\\
&\times\left[a_{1}E_{\beta}(kx^{\beta})+a_{2}E_{\beta}(-kx^{\beta})\right]\int\limits^{t}_{0}\tau^{\alpha-1}E_{\alpha,\alpha}((k^{2}+\delta)\tau^{\alpha})E_{\alpha}(k^{2}(t-\tau)^{\alpha})d\tau,
\end{aligned}
\end{eqnarray}
where $a_{i},k,\lambda,\gamma,\delta\in\mathbb{R}$ $(i=1,2,3,4)$. Observe that if $\alpha=\beta=1$, then equation \eqref{cc8} is exactly same as \eqref{cc7}. The above exact solution \eqref{cc8} for $k=a_1=a_2=a_3=a_4=\delta=\lambda=1$, $\gamma=-1$, $t=2$ and different values of $\alpha$ and $\beta$ is shown in Fig. (k) and Fig. (l).\\
\begin{figure}[h!]
\begin{center}
\begin{subfigure}[]{0.880\textwidth}
  \includegraphics[width=\textwidth]{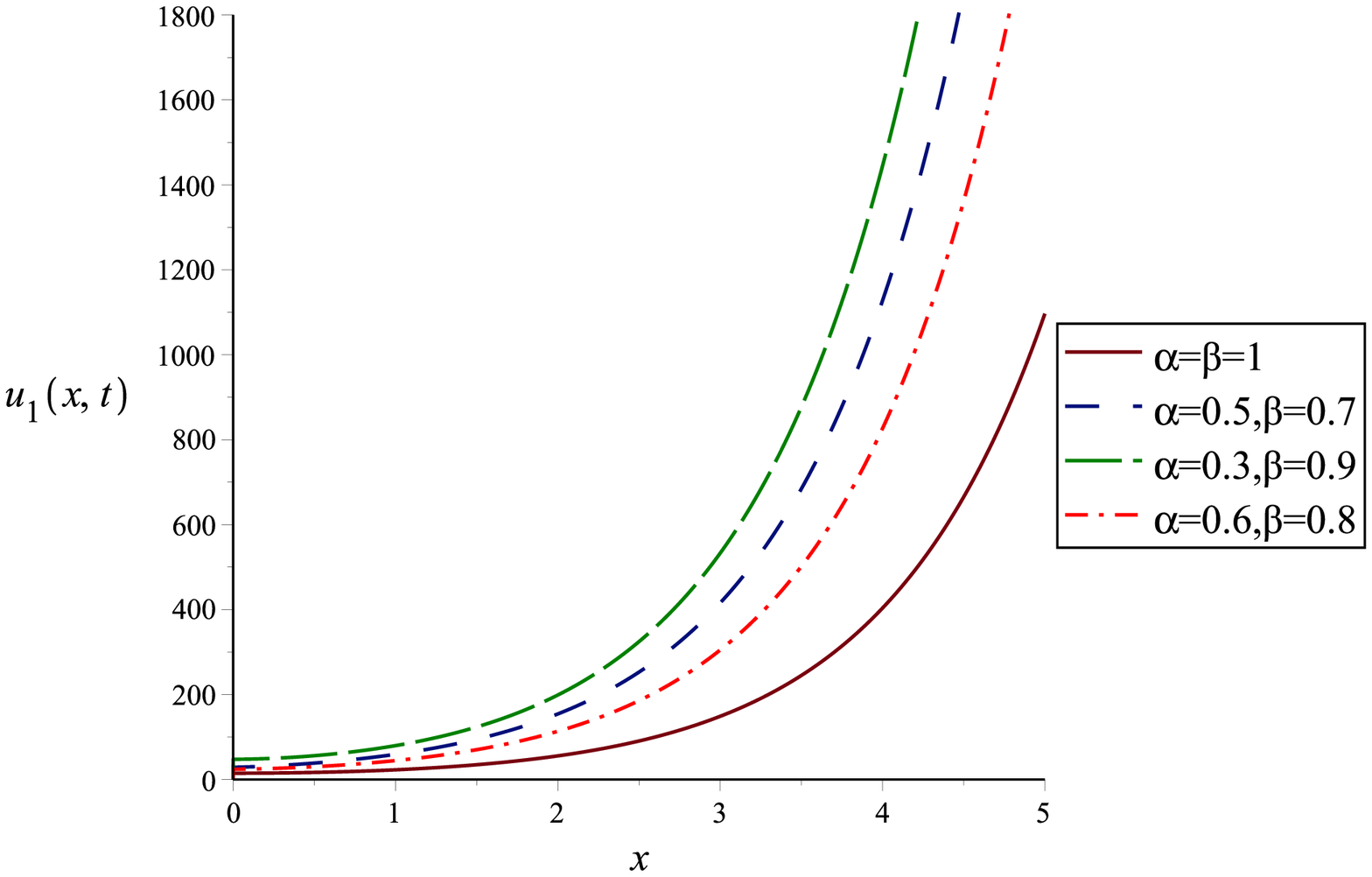}
\footnotesize{Fig.(k) Graphical representation of the solution $u_{1}(x,t)$ of \eqref{cc1} for $k=a_1=a_2=1$, $t=2$, and different values of $\alpha$ and $\beta$.}
 \end{subfigure}\hspace{10pt}
 \begin{subfigure}[]{0.880\textwidth}
  \includegraphics[width=\textwidth]{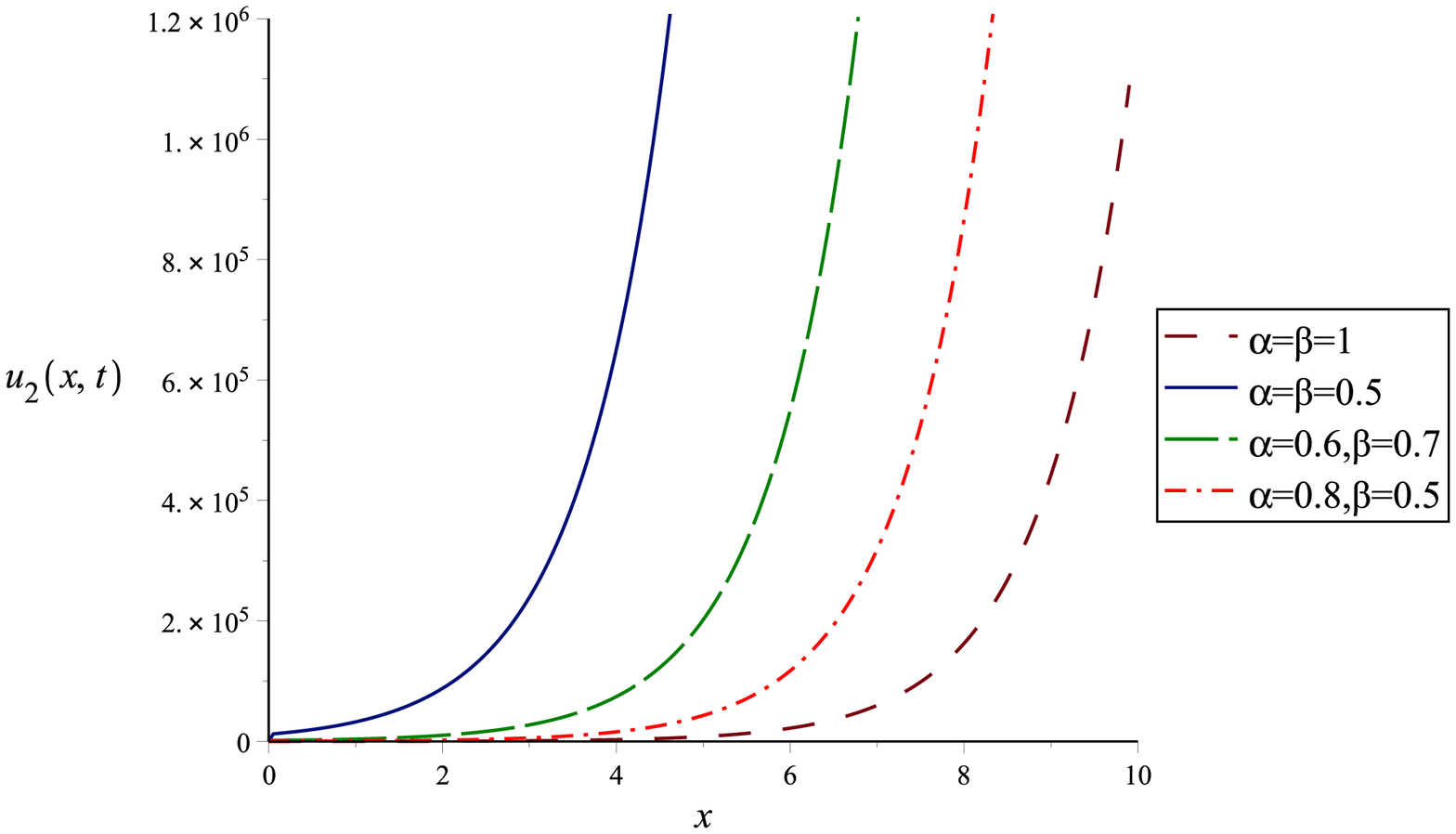}
 \footnotesize{Fig.(l) Graphical representation of the solution $u_{2}(x,t)$ of \eqref{cc1} for $k=a_1=a_2=a_3=a_4=\delta=\lambda=1$, $\gamma=-1$, $t=2$ and different values of $\alpha$ and $\beta$.}
 \end{subfigure}
\end{center}
\end{figure}
\textbf{Case 2.}\\
 Following the above similar procedure, we can derive another exact solution associated with the invariant subspace $\mathcal{W}_{2}^{1}\times \mathcal{W}_{2}^{2}=\mathfrak{L}\left\{1,x^{\beta}\right\}\times\mathfrak{L}\left\{1,x^{\beta}\right\}$. Let $\alpha=\beta=1$. In this case, we have
 \begin{eqnarray}\label{cc9}
\begin{aligned}
u_{1}(x,t)=&k_{1}+(\rho+\mu)k_{2}\left[\dfrac{k_{4}}{\delta}(e^{\delta t}-1)+\dfrac{\gamma k_{2}}{\delta^{2}}(e^{\delta t}-1-\delta t)\right]+k_{2}x,\\
u_{2}(x,t)=&k_{3}e^{\delta t}+\dfrac{k_{1}\gamma}{\delta}(e^{\delta t}-1)+\left[k_{4}e^{\delta t}+\dfrac{\gamma k_{2}}{\delta}\left(e^{\delta t}-1\right)\right]x\\
&+\dfrac{\gamma(\rho+\mu)k_{2}}{\delta^{2}}\left[k_{4}\left(\delta t e^{\delta t}-e^{\delta t}+1\right)+\dfrac{\gamma k_{2}}{\delta}\left(2-2e^{\delta t}+\delta t+\delta t e^{\delta t}\right)\right],
\end{aligned}
\end{eqnarray}
where $k_{i},(1=1,2,3,4),\gamma$ and $\delta(\neq0)$ are arbitrary constants.\\
While $\alpha,\beta\in(0,1]$, it takes
\begin{eqnarray}\label{cc10}
\begin{aligned}
u_{1}(x,t)=&k_{1}+(\Gamma(\beta+1))^{2}(\rho+\mu)k_{2}\left[k_{4}t^{\alpha}E_{\alpha,\alpha+1}(\delta t^{\alpha})+\gamma k_{2}t^{2\alpha}E_{\alpha,2\alpha+1}(\delta t^{\alpha})\right]+k_{2}x^{\beta},\\
u_{2}(x,t)=&k_{3}E_{\alpha}(\delta t^{\alpha})+\gamma k_{1}t^{\alpha}E_{\alpha,\alpha+1}(\delta t^{\alpha})+\left[k_{4}E_{\alpha}(\delta t^{\alpha})+\gamma k_{2}t^{\alpha}E_{\alpha,\alpha+1}(\delta t^{\alpha})\right]x^{\beta}\\
&+\left(\Gamma(\beta+1)\right)^{2}(\rho+\mu)k_{2}k_{4}\gamma\left[\int\limits^{t}_{0}\tau^{\alpha-1}E_{\alpha,\alpha}(\delta \tau^{\alpha})(t-\tau)^{\alpha}E_{\alpha,\alpha+1}(\delta(t-\tau)^{\alpha})d\tau\right]\\
&+\left(\Gamma(\beta+1)\right)^{2}(\rho+\mu)k_{2}^{2}\gamma^{2}\left[\int\limits^{t}_{0}\tau^{\alpha-1}E_{\alpha,\alpha}(\delta \tau^{\alpha})(t-\tau)^{2\alpha}E_{\alpha,2\alpha+1}(\delta(t-\tau)^{\alpha})d\tau\right],
\end{aligned}
\end{eqnarray}
where $k_{i},\gamma$, $\delta(\neq0)\in\mathbb{R}$ $(1=1,2,3,4)$.
Note that for $\alpha=\beta=1$, equations \eqref{cc10} are exactly same as \eqref{cc9}. The above exact solution \eqref{cc10} for different values of $\alpha$ and $\beta$ is shown in Fig. (m) and Fig. (n).

\begin{figure}[h!]
\begin{center}
\begin{subfigure}[]{0.880\textwidth}
  \includegraphics[width=\textwidth]{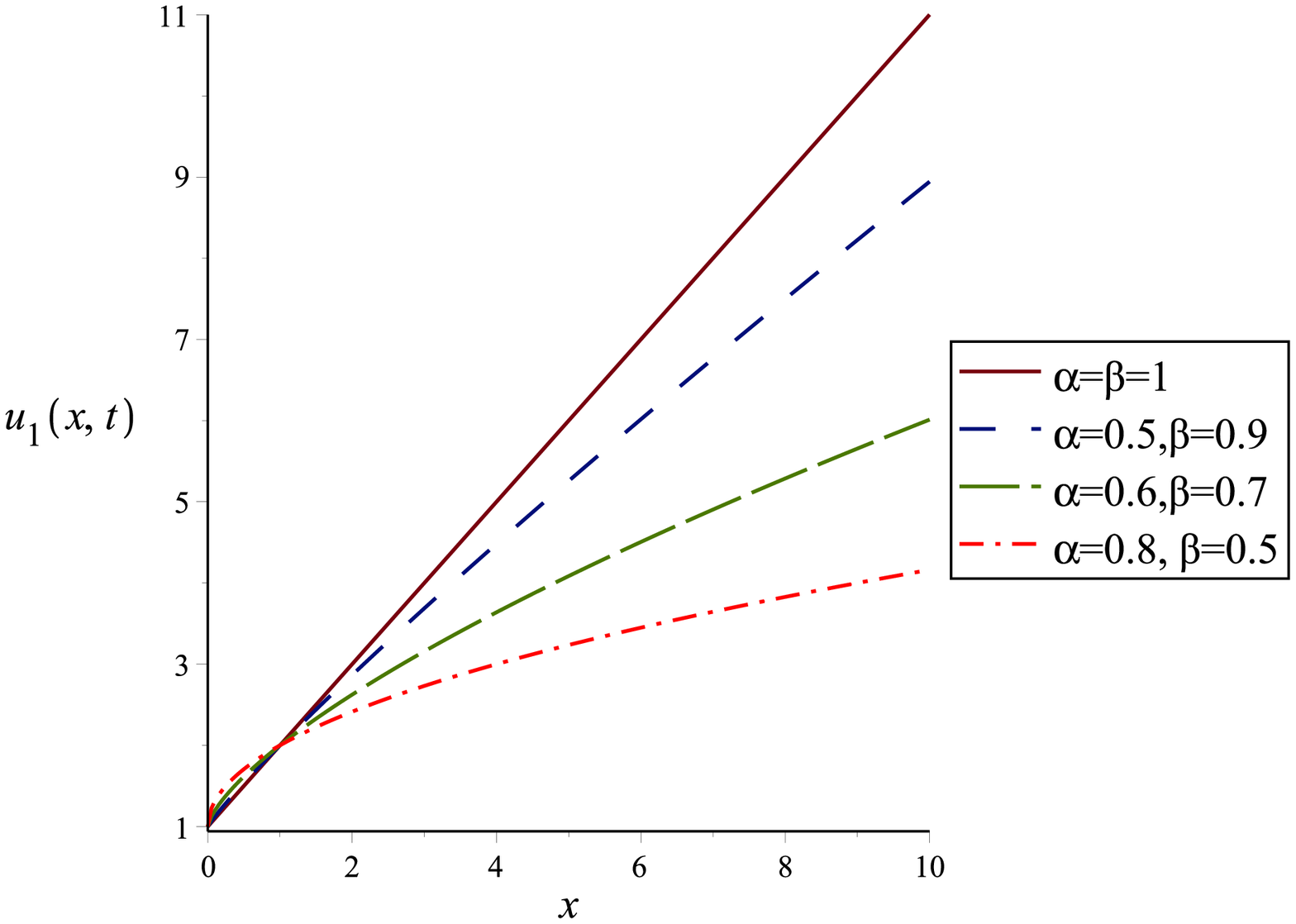}
\footnotesize{Fig.(m) Graphical representation of the solution $u_{1}(x,t)$ of \eqref{cc1} for $k_1=k_2=k_4=\rho=\delta=\gamma=1$, $\mu=-1$, $t=2$ and different values of $\alpha$ and $\beta$.}
 \end{subfigure}\hspace{10pt}
 \begin{subfigure}[]{0.880\textwidth}
  \includegraphics[width=\textwidth]{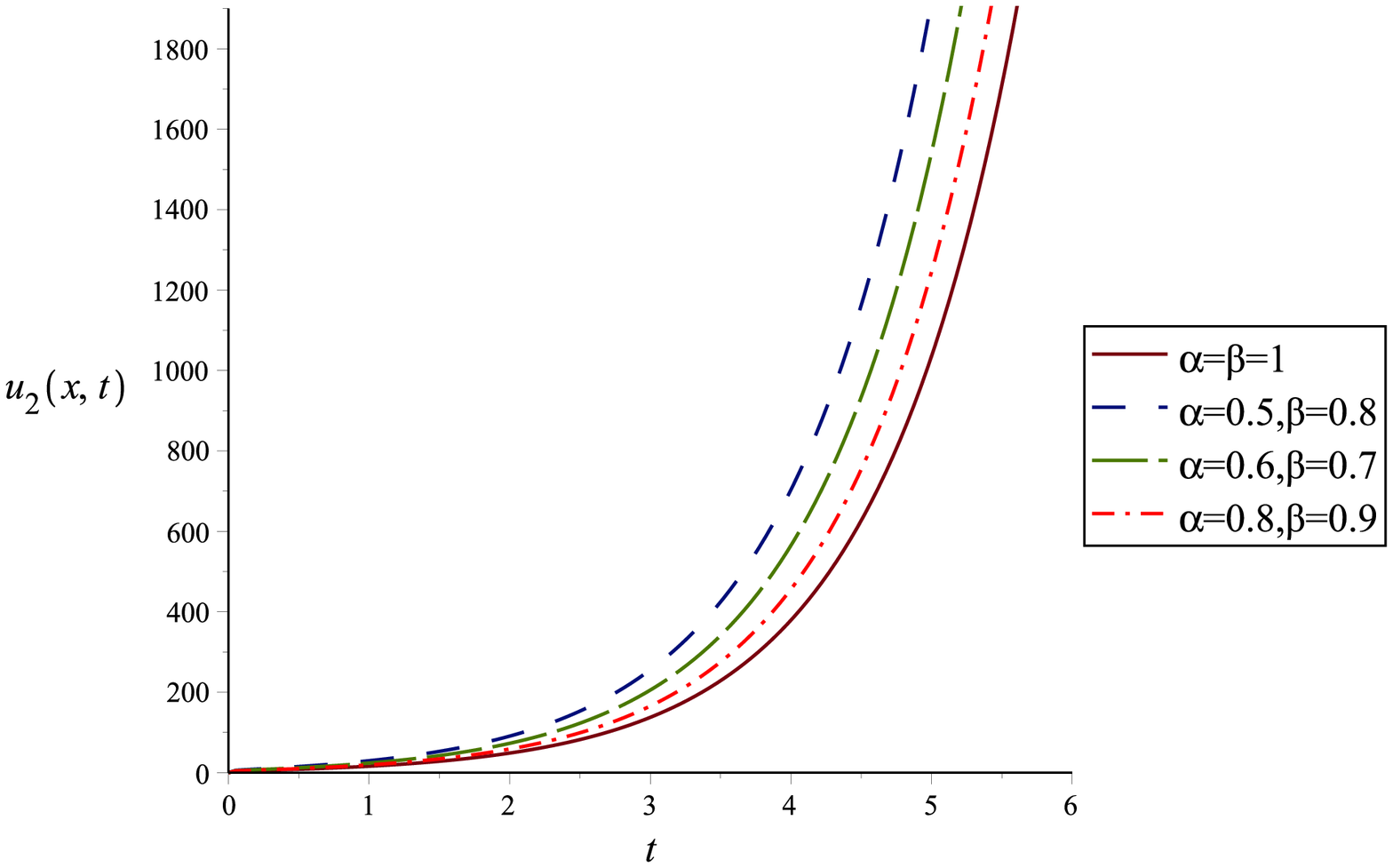}
 \footnotesize{Fig.(n) Graphical representation of the solution $u_{2}(x,t)$ of \eqref{cc1} for $k_1=k_2=k_3=k_4=\rho=\delta=\gamma=1$, $\mu=-1$, $x=2$, and different values of $\alpha$ and $\beta$.}
 \end{subfigure}
\end{center}
\end{figure}

\vspace{4cm}
\clearpage
\subsection{Two-coupled system of time-space fractional nonlinear model of stationary transonic plane-parallel gas flow}
Consider the following two-coupled system of time-space fractional PDEs 
\begin{eqnarray}
\begin{aligned}\label{eqc1}
\dfrac{\partial^{\alpha}u_{1}}{\partial t^{\alpha}}&=\dfrac{\partial^{\beta}u_{2}}{\partial x^{\beta}},\\
\dfrac{\partial^{\alpha}u_{2}}{\partial t^{\alpha}}&=-u_{1}\dfrac{\partial^{\beta}u_{1}}{\partial x^{\beta}},\ t>0,\ \alpha,\beta\in(0,1].
\end{aligned}
\end{eqnarray}
 which describes the model of stationary transonic plane-parallel gas flow \cite{gs,jef}. The symmetries of the above-coupled system \eqref{eqc1} with $\alpha=1$ and $\beta=1$ was presented in \cite{lv,jef}.
The system \eqref{eqc1} with $\alpha=\beta=1$ was discussed through the invariant subspace method by Galaktionov and Svirshchevskii \cite{gs}.

It is easy to find that \eqref{eqc1} admits an invariant subspace $\mathcal{W}_{2}^{1}\times \mathcal{W}_{2}^{2}=\mathfrak{L}\left\{1,x^{\beta}\right\}\times\mathfrak{L}\left\{1,x^{\beta}\right\}$ because
\begin{eqnarray*}
\hat{G}_{1}\left[A_{1}+A_{2}x^{\beta}, A_{3}+A_{4}x^{\beta}\right]&=&A_{4}\Gamma(\beta+1)\in W_{2}^{1},\\
\hat{G}_{2}\left[A_{1}+A_{2}x^{\beta}, A_{3}+A_{4}x^{\beta}\right]&=&-A_{2}\Gamma(\beta+1)\left[A_{1}+A_{2}x^{\beta}\right]\in W_{2}^{2}.
\end{eqnarray*}
Thus, we can write the exact solution of  equation \eqref{eqc1} in the form
\begin{equation}\label{eqc2}
u_{1}(x,t)=A_{1}(t)+A_{2}(t)x^{\beta},\ u_{2}(x,t)=A_{3}(t)+A_{4}(t)x^{\beta}.
\end{equation}
Let us first, consider $\alpha=\beta=1$. In this case, we yield an exact solution of the above system \eqref{eqc1} associated with $\mathcal{W}_{2}^{1}\times \mathcal{W}_{2}^{2}=\mathfrak{L}\left\{1,x^{\beta}\right\}\times\mathfrak{L}\left\{1,x^{\beta}\right\}$ as follows
\begin{eqnarray}
\begin{aligned}\label{eqc7}
u_{1}(x,t)&=k_{1}+k_{4}t-k_{2}^{2}\dfrac{t^{2}}{2}+k_{2}x,\\
u_{2}(x,t)&=k_{3}-k_{2}\left(k_{1}+\dfrac{k_{4}}{2}t-\dfrac{k_{2}^{2}}{6}t^{2}\right)t+\left(k_{4}-k_{2}^{2}t\right)x,\ k_{i}\in\mathbb{R}, i=1,2,3.4.
\end{aligned}
\end{eqnarray}
 Next, we consider $\alpha,\beta\in(0,1]$.
In this case, we have
\begin{eqnarray}
\begin{aligned}\label{eqcs}
u_{1}(x,t)=&k_{1}+k_{4}\Gamma(\beta+1)\dfrac{t^{\alpha}}{\Gamma(\alpha+1)}-k_{2}^{2}\left(\Gamma(\beta+1)\right)^{2}\dfrac{t^{2\alpha}}{\Gamma(2\alpha+1)}+k_{2}x^{\beta},\\
u_{2}(x,t)=&k_{3}-k_{1}k_{2}\Gamma(\beta+1)\dfrac{t^{\alpha}}{\Gamma(\alpha+1)}-k_{2}k_{4}\left(\Gamma(\beta+1)\right)^{2}\dfrac{t^{2\alpha}}{\Gamma(2\alpha+1)}\\
&+k_{2}^{3}\left(\Gamma(\beta+1)\right)^{3}\dfrac{t^{3\alpha}}{\Gamma(3\alpha+1)}+\left[{k_{4}}-k_{2}^{2}\Gamma(\beta+1)\dfrac{t^{\alpha}}{\Gamma(\alpha+1)}\right]x^{\beta},\ \alpha,\beta\in(0,1],
\end{aligned}
\end{eqnarray}
where $k_{1}, k_{2}, k_{3}$ and $k_{4}$ are non-zero arbitrary constants. Observe that for $\alpha=\beta=1$, equations \eqref{eqcs} are exactly same as equations \eqref{eqc7}. The above  exact solution \eqref{eqcs} for $k_1=2$, $k_2=k_3=k_4=1$, $x=2$, and different values of $\alpha$ and $\beta$ is shown in Fig. (o) and Fig. (p).
\begin{figure}[h!]
\begin{center}
\begin{subfigure}[]{0.94\textwidth}
  \includegraphics[width=\textwidth]{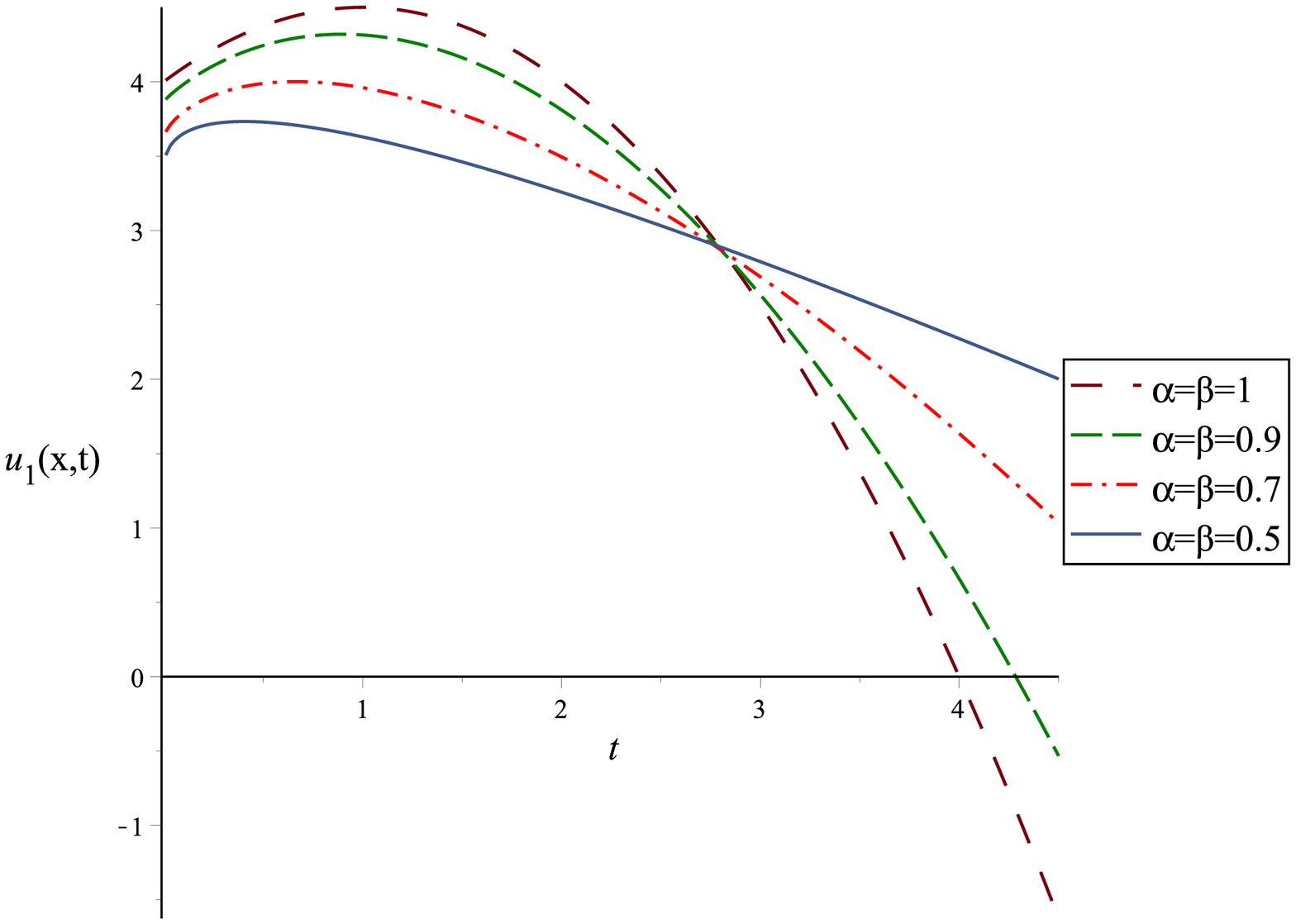}
\footnotesize{Fig.(o) Graphical representation of the solution $u_{1}(x,t)$ of \eqref{eqc1} for $k_1=2$, $k_2=k_4=1$, $x=2$, and different values of $\alpha$ and $\beta$.}
 \end{subfigure}\hspace{10pt}
 \begin{subfigure}[]{0.94\textwidth}
  \includegraphics[width=\textwidth]{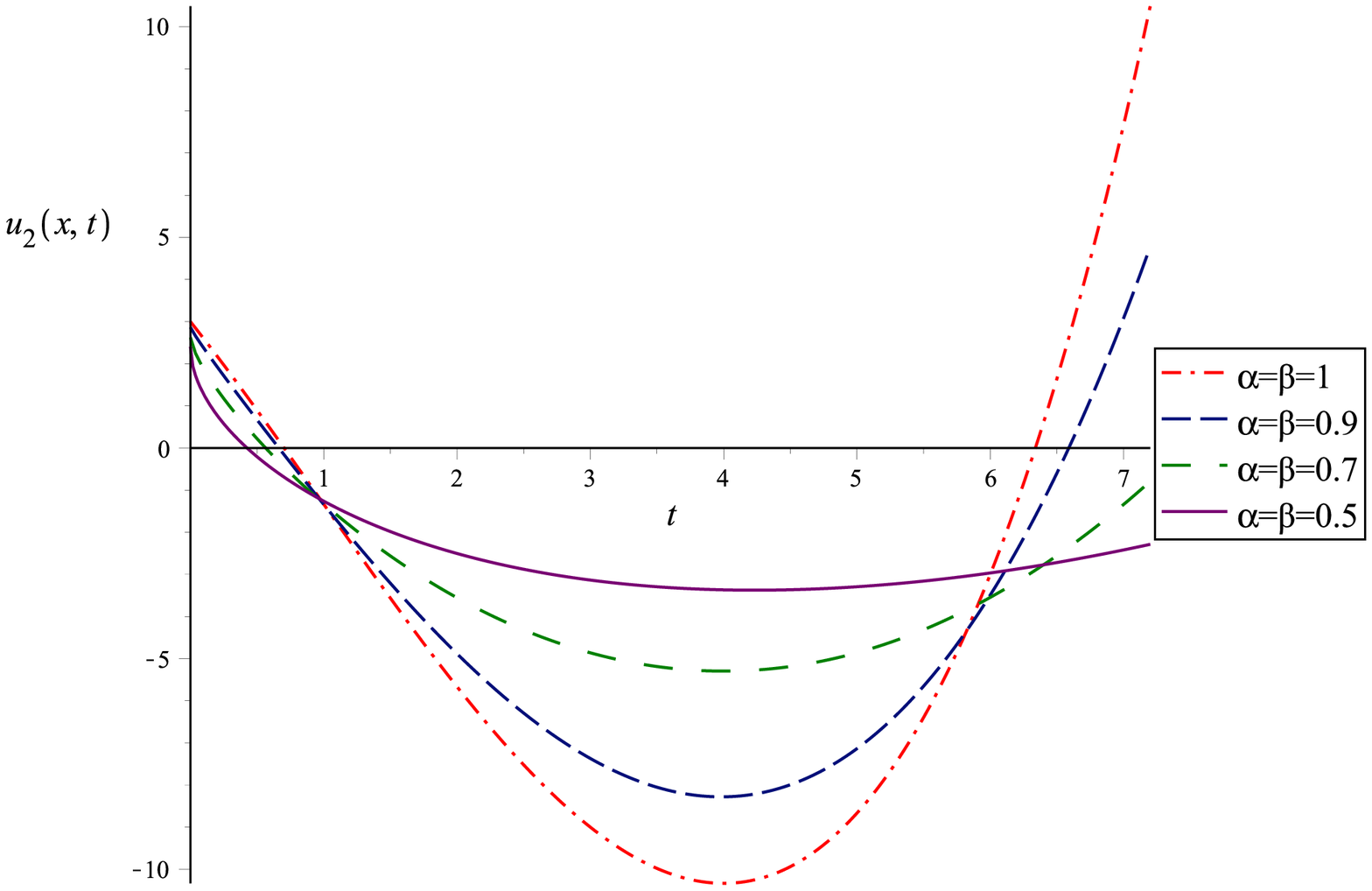}
 \footnotesize{Fig.(p) Graphical representation of the solution $u_{2}(x,t)$ of \eqref{eqc1} for $k_1=2$, $k_2=k_3=k_4=1$, $x=2$, and different values of $\alpha$ and $\beta$.}
 \end{subfigure}
\end{center}
\end{figure}

\subsection{Three-coupled system of time-space fractional generalized Hirota-Satsuma-KdV equation}
Consider the following three-coupled system of time-space fractional generalized Hirota-Satsuma-KdV equation
\begin{eqnarray}\label{gkdv}
\begin{aligned}
\dfrac{^{RL}\partial^{\alpha}u_{1}}{\partial t^{\alpha}}&=\dfrac{1}{2}\dfrac{\partial^{\beta+2}u_{1}}{\partial x^{\beta+2}}-3 u_{1}\dfrac{\partial^{\beta} u_{1}}{\partial x^{\beta}}+3u_{3}\dfrac{\partial^{\beta}u_{2}}{\partial x^{\beta}}+3u_{2}\dfrac{\partial^{\beta}u_{3}}{\partial x^{\beta}},\\
\dfrac{^{RL}{\partial}^{\alpha}u_{2}}{\partial t^{\alpha}}&=-\dfrac{\partial^{\beta+2}u_{2}}{\partial x^{\beta+2}}+3 u_{1}\dfrac{\partial^{\beta} u_{2}}{\partial x^{\beta}},\\
\dfrac{^{RL}{\partial}^{\alpha}u_{3}}{\partial t^{\alpha}}&=-\dfrac{\partial^{\beta+2}u_{3}}{\partial x^{\beta+2}}+3 u_{1} \dfrac{\partial^{\beta} u_{3}}{\partial x^{\beta}},\   \ t>0,  \alpha,\beta\in(0,1].
\end{aligned}
\end{eqnarray}
This system with $\alpha=\beta=1$, describes an interaction of two long waves with different dispersion relations \cite{gkd}. The above system \eqref{gkdv} with $\beta=1$ was discussed through the invariant subspace method in \cite{pra2}.
 Let us consider the dimension $n_{1}=n_{2}=n_{3}=2$. For this case, the equation \eqref{gkdv} admits the following invariant subspace $$\mathcal{W}_{2,2,2}=\mathcal{W}^{1}_{2}\times\mathcal{W}^{2}_{2}\times\mathcal{W}^{3}_{2}=\mathfrak{L}\left\{1,x^{\beta}\right\}\times\mathfrak{L}\left\{1,x^{\beta}\right\}\times\mathfrak{L}\left\{1,x^{\beta}\right\}.$$
Thus, we can write the exact solution of system \eqref{gkdv} in the form
\begin{eqnarray}
\begin{aligned}\label{gkdv1}
u_{1}(x,t)&=A_{1}(t)+A_{2}(t)x^{\beta},\\
u_{2}(x,t)&=A_{3}(t)+A_{4}(t)x^{\beta},\\
u_{3}(x,t)&=A_{5}(t)+A_{6}(t)x^{\beta},
\end{aligned}
\end{eqnarray}
where $A_{i}(t)$, $(i=1,\dots,6)$ are satisfy the following system of FODEs
\begin{eqnarray}
\begin{aligned}\label{111}
\dfrac{d^{\alpha}A_{1}(t)}{dt^{\alpha}}&=3\Gamma(\beta+1)\left[-A_{1}(t)A_{2}(t)+A_{3}(t)A_{6}(t)+A_{4}(t)A_{5}(t)\right],\\
\dfrac{d^{\alpha}A_{2}(t)}{dt^{\alpha}}&=3\Gamma(\beta+1)\left[-A_{2}^{2}(t)+2A_{4}(t)A_{6}(t)\right],\\
\dfrac{d^{\alpha}A_{3}(t)}{dt^{\alpha}}&=3\Gamma(\beta+1)A_{1}(t)A_{4}(t),\\
\dfrac{d^{\alpha}A_{4}(t)}{dt^{\alpha}}&=3\Gamma(\beta+1)A_{2}(t)A_{4}(t),\\
\dfrac{d^{\alpha}A_{5}(t)}{dt^{\alpha}}&=3\Gamma(\beta+1)A_{1}(t)A_{6}(t),\\
\dfrac{d^{\alpha}A_{6}(t)}{dt^{\alpha}}&=3\Gamma(\beta+1)A_{2}(t)A_{6}(t).
\end{aligned}
\end{eqnarray}
Solving the above equations \eqref{111}, we have
\begin{eqnarray}
\begin{aligned}\label{3s2}
u_{1}(x,t)&=\dfrac{1}{3\Gamma(\beta+1)}\left(\dfrac{\Gamma(1-\alpha)}{\Gamma(1-2\alpha)}\right)\left[\dfrac{k_{1}}{k_{2}}+x^{\beta}\right]t^{-\alpha},\\
u_{2}(x,t)&=\dfrac{1}{9k_{2}(\Gamma(\beta+1))^{2}}\left(\dfrac{\Gamma(1-\alpha)}{\Gamma(1-2\alpha)}\right)^{2}\left[\dfrac{k_{1}}{k_{2}}+x^{\beta}\right]t^{-\alpha},\\
u_{3}(x,t)&=k_{1}t^{-\alpha}+k_{2}x^{\beta}t^{-\alpha},\ k_{2}\neq0,\ \alpha\in(0,1)-\left\{\frac{1}{2}\right\},\ \beta\in(0,1],\ k_1,k_2\in\mathbb{R}.
\end{aligned}
\end{eqnarray}
The above exact solution \eqref{3s2} for $k_1=2$, $k_2=1$, $x=2$, and different values of $\alpha$ and $\beta$ is shown in Fig. (q), Fig. (r) and Fig. (s).
 We would like to mention that when $\beta=1$, the above solution \eqref{3s2} is exactly the same as given in \cite{pra2}.

\begin{figure}[h!]
\begin{center}
\begin{subfigure}[]{0.48\textwidth}
  \includegraphics[width=\textwidth]{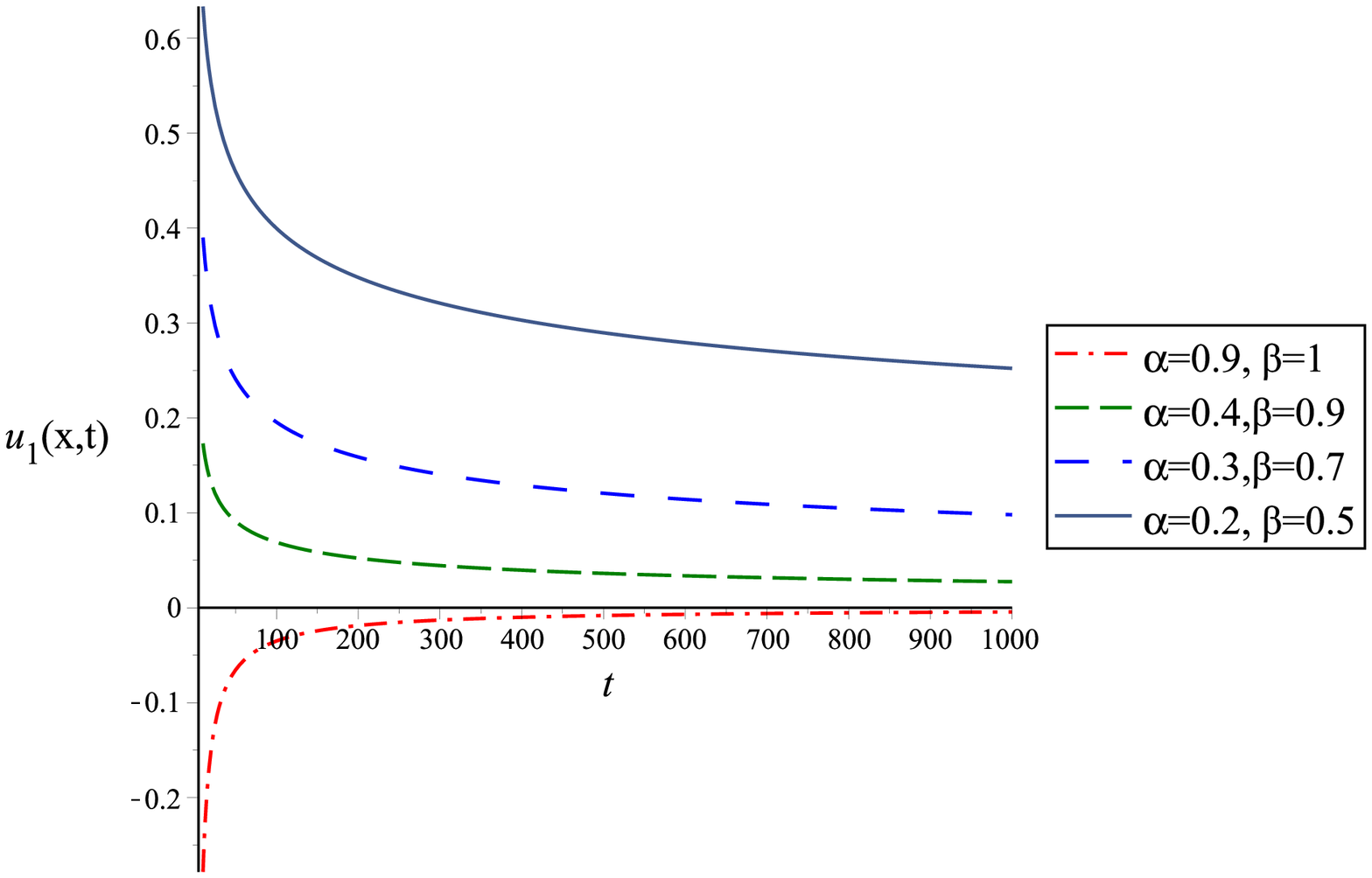}
\footnotesize{Fig.(q) Graphical representation of the solution $u_{1}(x,t)$ of \eqref{gkdv} for $k_1=2$, $k_2=1$, $x=2$, and different values of $\alpha$ and $\beta$.}
 \end{subfigure}\hspace{10pt}
 \begin{subfigure}[]{0.48\textwidth}
  \includegraphics[width=\textwidth]{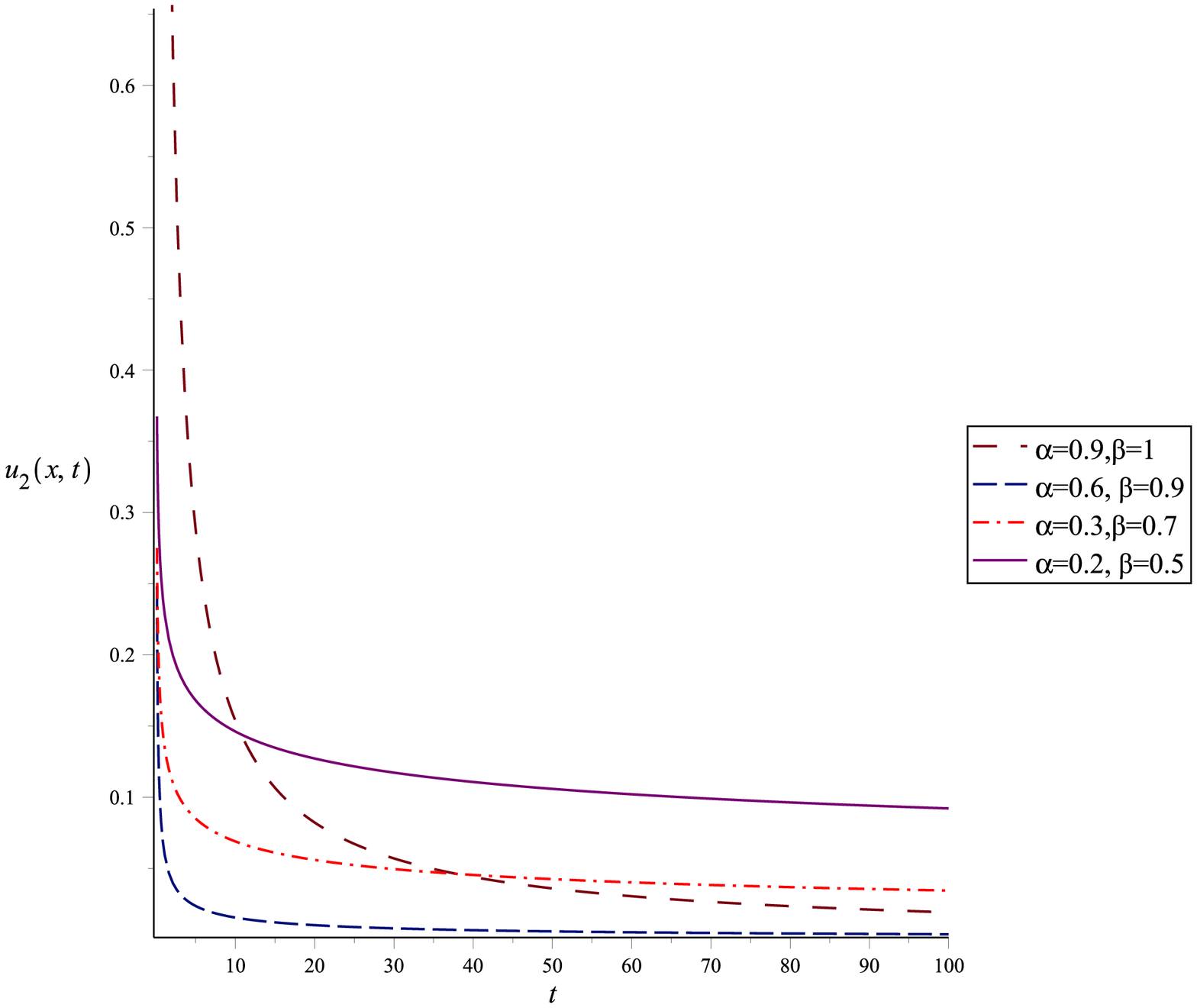}
 \footnotesize{Fig.(r) Graphical representation of the solution $u_{2}(x,t)$ of \eqref{gkdv} for $k_1=2$, $k_2=1$, $x=2$, and different values of $\alpha$ and $\beta$.}
 \end{subfigure}\hspace{10pt}
 \begin{subfigure}[]{0.93\textwidth}
  \includegraphics[width=\textwidth]{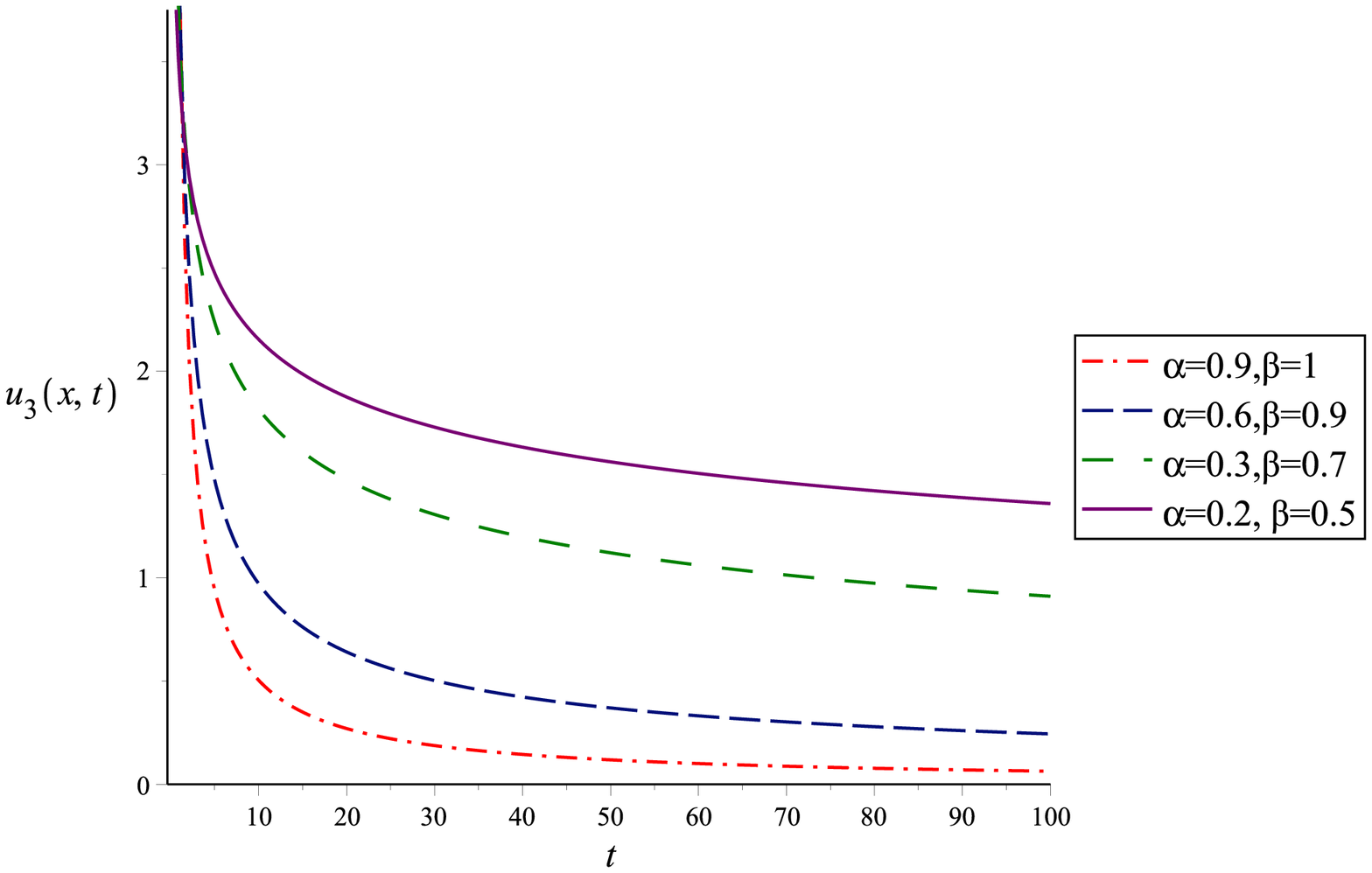}
 \footnotesize{Fig.(s) Graphical representation of the solution $u_{3}(x,t)$ of \eqref{gkdv} for $k_1=2$, $k_2=1$, $x=2$, and different values of $\alpha$ and $\beta$.}
 \end{subfigure}
\end{center}
\end{figure}

\begin{nt}
We would like to point out that the property \eqref{pp2} holds only for $\mu>n-1$. Hence $\dfrac{d^{\alpha}(t^{-\alpha})}{dt^{\alpha}}=\dfrac{\Gamma(1-\alpha)}{\Gamma(1-2\alpha)}t^{-2\alpha}$ is not valid in Caputo fractional derivative of order $\alpha\in(0,1)$, but it holds for R-L fractional derivative of order $\alpha\in(0,1)$:
$$\dfrac{d^{\alpha}(t^{-\alpha})}{dt^{\alpha}}=\dfrac{\Gamma(1-\alpha)}{\Gamma(1-2\alpha)}t^{-2\alpha},\ \mu=-\alpha>-1.$$
\end{nt}
\section{Conclusion}
In this article, we have presented how the invariant subspace method can be extended to a scalar and coupled system of time-space FPDEs. Also, we have explicitly presented how the scalar and coupled system of time-space FPDEs admit more than one invariant subspace which in turn helps to derive more than one exact solution. The applicability of the method was illustrated through scalar and coupled system of time-space FPDEs given in \eqref{doce}, \eqref{eqsr1}, \eqref{DS1}, \eqref{cc1}, \eqref{eqc1} and \eqref{gkdv}.
Using the invariant subspace method, the scalar and coupled system of time-space FPDEs are reduced to the system of FODEs. The obtained reduced system of FODEs can be solved by well-known analytical methods.
The obtained exact solutions can be expressed in terms of the polynomial and well-known Mittag-Leffler functions.
These investigations show that the invariant subspace method is a very effective tool to derive exact solutions for the scalar and coupled system of nonlinear time-space FPDEs.
\section*{Acknowledgments}
The author also would like to thank R. Sahadevan, UGC Emeritus Professor, Ramanujan Institute for Advanced Study in Mathematics, University of Madras, Chennai, India, for his helpful comments and suggestions for significant improvement of the manuscript.

\end{document}